\documentclass[10pt,a4paper]{amsart}

\RequirePackage{etex}
\usepackage[utf8]{inputenc}
\usepackage[T1]{fontenc}
\usepackage[english]{babel}

\usepackage[mathscr]{euscript}
\usepackage{caption}
\usepackage{subcaption}
\usepackage{datetime}
\usepackage{url}
\usepackage{verbatim}
\usepackage{float}
\usepackage{lipsum}

\usepackage{titletoc}
\usepackage{textcomp}
\usepackage{anysize}
\usepackage{fancyhdr}
\usepackage{calrsfs}

\usepackage{amssymb}
\usepackage{amsmath,mathtools}
\usepackage{amsfonts}
\usepackage{amscd}
\usepackage{amsthm}
\usepackage{dsfont}

\usepackage{latexsym}
\usepackage{enumerate}
\usepackage[shortlabels]{enumitem}
\usepackage{array}
\usepackage{stackrel}
\usepackage{stmaryrd}
\usepackage{stackengine}

\usepackage{graphicx}
\graphicspath{ {Images/} }
\usepackage{color,graphics}
\usepackage{tikz}
	\usetikzlibrary{arrows,automata,patterns,decorations.pathmorphing,decorations.pathreplacing, decorations.markings,shapes.misc}
\usepackage{tikz-cd}
	\usetikzlibrary{matrix,arrows,calc,automata,trees}
\usepackage[all,cmtip]{xy}
\usepackage{pgfplots,caption}
\pgfplotsset{compat=1.9}

\usepackage{lscape}
\usepackage[active]{srcltx}
\usepackage{rotating}

\usepackage[hypertexnames=false]{hyperref}

\makeatletter

\renewcommand{\tocsection}[3]{
  \indentlabel{\@ifnotempty{#2}{\bfseries\ignorespaces#1 #2\quad}}\bfseries#3}
\renewcommand{\tocsubsection}[3]{
  \indentlabel{\@ifnotempty{#2}{\ignorespaces#1 #2\quad}}#3}

\newcommand\@dotsep{4.5}
\def\@tocline#1#2#3#4#5#6#7{\relax
  \ifnum #1>\c@tocdepth 
  \else
    \par \addpenalty\@secpenalty\addvspace{#2}%
    \begingroup \hyphenpenalty\@M
    \@ifempty{#4}{%
      \@tempdima\csname r@tocindent\number#1\endcsname\relax
    }{%
      \@tempdima#4\relax
    }%
    \parindent\z@ \leftskip#3\relax \advance\leftskip\@tempdima\relax
    \rightskip\@pnumwidth plus1em \parfillskip-\@pnumwidth
    #5\leavevmode\hskip-\@tempdima{#6}\nobreak
    \leaders\hbox{$\m@th\mkern \@dotsep mu\hbox{.}\mkern \@dotsep mu$}\hfill
    \nobreak
    \hbox to\@pnumwidth{\@tocpagenum{\ifnum#1=1\bfseries\fi#7}}\par
    \nobreak
    \endgroup
  \fi}
\AtBeginDocument{%
\expandafter\renewcommand\csname r@tocindent0\endcsname{0pt}
}
\def\l@subsection{\@tocline{2}{0pt}{0.9cm}{5pc}{}}

\newcases{lrcases}
  {\quad}
  {$\m@th\displaystyle{##}$\hfil}
  {$\m@th\displaystyle{##}$\hfil}
  {\lbrace}
  {\rbrace}
\newcases{lrcases*}
  {\quad}
  {$\m@th\displaystyle{##}$\hfil}
  {{##}\hfil}
  {\lbrace}
  {\rbrace}

\makeatother

\fancypagestyle{my_own_style}{
		\fancyhf{}
		\fancyfoot[C]{Page \thepage}
}

\newcommand{\N}{{\mathbb N}}
\newcommand{\Z}{{\mathbb Z}}
\newcommand{\Q}{{\mathbb Q}}
\newcommand{\C}{{\mathbb C}}

\DeclareMathAlphabet{\pazocal}{OMS}{zplm}{m}{n}

\newcommand{\calA}{{\pazocal A}}
\newcommand{\calB}{{\pazocal B}}

\newcommand{\calD}{{\pazocal D}}

\newcommand{\calK}{{\pazocal K}}
\newcommand{\calL}{{\pazocal L}}

\newcommand{\calN}{{\pazocal N}}

\newcommand{\calQ}{{\pazocal Q}}
\newcommand{\calR}{{\pazocal R}}

\newcommand{\calU}{{\pazocal U}}

\newcommand{\calZ}{{\pazocal Z}}

\newcommand{\ra}{\rightarrow}

\newcommand{\ol}{\overline}

\newcommand{\floor}[1]{\lfloor #1\rfloor}

\numberwithin{equation}{section}

\theoremstyle{plain}

\newtheorem{theorem}{Theorem}[section]
\newtheorem*{theorem*}{Theorem}
\newtheorem{lemma}[theorem]{Lemma}
\newtheorem{proposition}[theorem]{Proposition}
\newtheorem*{proposition*}{Proposition}
\newtheorem{corollary}[theorem]{Corollary}

\theoremstyle{definition}

\newtheorem{definition}[theorem]{Definition}

\newtheorem{remark}[theorem]{Remark}
\newtheorem*{remark*}{Remark}
\newtheorem{remarks}[theorem]{Remarks}
\newtheorem*{remarks*}{Remarks}

\newtheorem*{assumption*}{Assumption}

\newtheorem{notation}[theorem]{Notation}
\newtheorem{claim}[theorem]{Claim}
\newtheorem*{claim*}{Claim}

\tikzset{
  on each segment/.style={
    decorate,
    decoration={
      show path construction,
      moveto code={},
      lineto code={
        \path [#1]
        (\tikzinputsegmentfirst) -- (\tikzinputsegmentlast);
      },
      curveto code={
        \path [#1] (\tikzinputsegmentfirst)
        .. controls
        (\tikzinputsegmentsupporta) and (\tikzinputsegmentsupportb)
        ..
        (\tikzinputsegmentlast);
      },
      closepath code={
        \path [#1]
        (\tikzinputsegmentfirst) -- (\tikzinputsegmentlast);
      },
    },
  },
  mid arrow/.style={postaction={decorate,decoration={
        markings,
        mark=at position .5 with {\arrow[#1]{stealth}}
      }}},
}

\tikzset{
	cross/.style={cross out, draw=black, minimum size=2*(#1-\pgflinewidth), inner sep=0pt, outer sep=0pt},
	cross/.default={1pt}
}

\tikzset{
  symbol/.style={
    draw=none,
    every to/.append style={
      edge node={node [sloped, allow upside down, auto=false]{$#1$}}}
  }
}

\title[units group rings]{On group rings of virtually abelian groups}
\author{Joan Claramunt}
\address[J. Claramunt]{Fakultät für Mathematik und Informatik, Universität Leipzig, Augustusplatz 10, 04109 Leipzig, Germany.}
\email{Joan.Claramunt@math.uni-leipzig.de}
\author{Łukasz Grabowski}
\address[Ł. Grabowski]{Fakultät für Mathematik und Informatik, Universität Leipzig, Augustusplatz 10, 04109 Leipzig, Germany.}
\email{Lukasz.Grabowski@math.uni-leipzig.de}

\subjclass[2010]{Primary 16S34, 20C07; Secondary 16E50}
\keywords{}

\thanks{Both authors were supported by the ERC Starting Grant ``Limits of Structures in Algebra and Combinatorics'' No. 805495.}

\date{\today}

\begin{document}

\pagestyle{plain}
 
\begin{abstract}
Let $\Gamma$ be a finitely generated torsion-free group. We show that the statement of $\Gamma$ being virtually abelian is equivalent to the statement that the $*$-regular closure of the group ring $\C[\Gamma]$ in the algebra of (unbounded) operators affiliated to the group von Neumann algebra is a central division algebra. More generally, for any field $k$, it is shown that $k[\Gamma]$ embeds into a central division algebra in case $\Gamma$ is virtually abelian. We take advantage of this result in order to develop a criterion for existence of units in the group ring $k[\Gamma]$. We develop this criterion in the particular case of $\Gamma$ being the Promislow's group.
\end{abstract}

\maketitle

\tableofcontents

\normalsize

\section{Introduction}\label{section-introduction}

Recently Gardam \cite{Gardam} presented a counterexample to one of the famous conjectures on group rings attributed to Kaplansky: the unit conjecture. This conjecture asserts that the only units in the group ring of a torsion-free group $\Gamma$, with coefficients in a field $k$, are the trivial ones, namely the ones of the form $\lambda g$ with $\lambda \in k \backslash \{0\}, g \in \Gamma$. Gardam disproved this conjecture by explicitly exhibiting a group $G$, a field $k$ and the non-trivial unit. His field is the field $k = \mathbb{F}_2$ of order $2$, and the group $G$ is the Promislow's group, the unique torsion-free $3$-dimensional crystallographic group with finite abelianization. Later on, Murray \cite{Murray} disproved the conjecture by using the same grup but over fields of arbitrary positive characteristic $\text{char}(k) > 0$ following the same lines as in \cite{Gardam}. 

The unit conjecture is still open when the coefficient field is a field of characteristic zero. While we develop some interesting new criteria for the existence of units in this case, we have not been able to achieve our initial goal of settling this conjecture, even in the case of the Promislow's group.

The initial point of our investigation is the following observation. Let $\calR(\Gamma)$  be the $*$-regular closure of the group ring $\C[\Gamma]$ inside the algebra $\calU(\Gamma)$ of (possibly unbounded) operators affiliated to the group von Neumann algebra $\calN(\Gamma)$. When $\Gamma$ is torsion-free we may equivalently define $\calR(\Gamma)$ to be the Ore localization of $\C[\Gamma]$. Then $\calR(\Gamma)$ is a central division algebra, i.e. it is a division algebra which is finite-dimensional over its center. With this in mind, the problem of finding units which lie specifically in $\C[\Gamma]$ has a flavor which is very similar to the problem of determining the units in the ring of integers of a finite field extension of the rational numbers.

In the case of the Promislov's group, the division algebra $\calR(\Gamma)$ is $16$-dimensional over its center, and the center is a $2$-dimensional extension of the field of rational functions in $3$ variables. First let us note that, irrespectively of the unit conjecture, it is interesting to ask which groups $\Gamma$ have the property that $\calR(\Gamma)$ is a central division algebra. The following theorem answers this question. 

\begin{theorem*}[Theorem \ref{theorem-characterization.virtually.abelian}]\label{theorem-main.theorem.1}
If $\Gamma$ is a finitely generated torsion-free group, then $\Gamma$ is virtually abelian if and only if the $*$-regular closure $\calR(\Gamma) := \calR(\C[\Gamma],\calU(\Gamma))$ of $\C[\Gamma]$ in $\calU(\Gamma)$ is a  central division algebra.
\end{theorem*}

Let us now briefly describe how we develop a criterion for the existence of units in $\C[\Gamma]$ in the general case when $\Gamma$ is an arbitrary virtually abelian group. The details are contained in Section \ref{section-general.procedure}. Suppose we are given a polynomial $p(x)$ of degree $d \geq 2$ with coefficients in $\text{Z}(k[\Gamma])$, the center of $k[\Gamma]$. Suppose further that there exists a root of $p(x)$ in $\text{Z}(k[\Gamma])$, call it $\kappa$. We can then write $p(x)$ as the product $(x-\kappa)q(x)$ for some polynomial $q(x)$ with coefficients also in $\text{Z}(k[\Gamma])$. Then if $T \in k[\Gamma]$ is any element such that $p(T) = 1$, we obtain $$1 = (T-\kappa) q(T),$$ so $T - \kappa$ is a unit in $k[\Gamma]$, which may be trivial or not. 

Let us focus on the case when $\Gamma$ is the Promislow's group $G$. First, let us make some preliminary observations about units in
$k[G]$ when $\text{char}(k) \neq 2$. As we show in Propositions \ref{proposition-two.groups} and \ref{proposition-omega.two.groups},  the non-trivial units in $k[G]$ can be classified into two types. Let us write $\calZ$ for the center of $\calR(\Gamma)$ and let $U$ be a non-trivial unit in $k[G]$. Write $\omega^U = a + bs + ct + du \in k[G]$, where $a,b,c,d \in k[\Z^3]$, and $s,t,u$ are the standard generator of $G$. Then both of the fields $\calZ(a+bs)$ and $\calZ(ct+du)$ have the same degree over $\calZ$ and this common degree is equal to either $2$ or $4$. We say that $U$ is of type $2$ in the first case and of type $4$ in the second case. A direct check shows that the units found by Murray in \cite{Murray} when $\text{char}{(k)}>2$ are of type $2$ (Proposition \ref{proposition-omega.22.examples}), and as such it makes sense to ask about criteria for existence of units of type $2$ when $\text{char}{(k)} = 0$.

With this in mind, let us state our criterion for the existence of units of type $2$ in $\C[G]$.

\begin{theorem*}[Theorem \ref{theorem-criterion.unit}]
Let $G$ be the Promislow group. Then there exist non-trivial units of type 2 in $\C[G]$ if and only if the following system of quadratic equation in $9$ variables with coefficients in $\calZ$ has non-trivial solutions in $\calZ$:
\begin{align*}
\diamond \text{ } & (A^2-1)\omega_x^2 + C_1^2(\kappa_x+1)(\kappa_y-1)-C_2^2(\kappa_x-1)(\kappa_y+1) + C_3^2(\kappa_x+1)(\kappa_y+1)-C_4^2(\kappa_x-1)(\kappa_y-1) \\
& \phantom{(A^2-1)\omega_x^2} + 2D_1^2(\kappa_z-1)-2D_2^2\omega_x^2(\kappa_z+1) + 2D_3^2(\kappa_z+1)-2D_4^2\omega_x^2(\kappa_z-1) = 0, \\
\diamond \text{ } & (\kappa_x-1)C_2C_4 - (\kappa_x+1)C_1C_3 = 0, \\
\diamond \text{ } & (\kappa_x-1)C_2D_2 + C_3D_3 = 0, \\
\diamond \text{ } & (\kappa_x+1)C_1D_2 + C_4D_3 = 0, \\
\diamond \text{ } & (\kappa_x-1)C_4D_4 + C_1D_1 = 0, \\
\diamond \text{ } & (\kappa_x+1)C_3D_4 + C_2D_1 = 0.
\end{align*}
Here $\kappa_g = \frac{1}{2}(g+g^{-1}) \in \C[G]$ and $\omega_g = \frac{1}{2}(g-g^{-1}) \in \C[G]$ for $g \in G$.
\end{theorem*}

This criterion should be compared with the determinant condition stated in \cite{determinant}, which gives a single polynomial equation of degree $4$ with $16$ variables. In this sense, our criterion gives a `simpler' polynomial equation than the determinant condition.

\begin{remark*}
It is not known if the group ring $\C[G]$ of the Promislow's group contains any non-trivial unitaries. In Theorem \ref{theorem-criterion.unitary} we state an if-and-only-if criterion similar to the theorem above for the existence of unitaries.
\end{remark*}

The paper is organized as follows. In Section \ref{section-preliminaries} we collect the definitions of von Neumann regular and $*$-regular rings, together with the definition of the $*$-regular closure of a subring in a $*$-regular ring. We also recall the construction of a crossed product of a ring with a group, from which the group ring of a ring is a particular case of it. In Section \ref{section-virtually.abelian} we present and prove the first of our main results, namely the characterization of virtually abelian groups $\Gamma$ in terms of algebraic properties of the $*$-regular closure of the group ring $\C[\Gamma]$ inside $\calU(\Gamma)$, the algebra of (unbounded) operators affiliated to the group von Neumann algebra $\calN(\Gamma)$. We use part of this characterization, in Section \ref{section-general.procedure}, to briefly present a general procedure for finding units in group algebras of virtually abelian groups. We specialize, and further analyze, this procedure in the case $\Gamma$ is the Promislow's group $G$, culminating in the criterias presented in Theorems \ref{theorem-criterion.unit} and \ref{theorem-criterion.unitary}, which give necessary and sufficient conditions for finding non-trivial units (resp. unitaries) in the group ring of $G$. This is done in Section \ref{section-criterion.promislow}.

\section{Preliminaries}\label{section-preliminaries}

\subsection{Central division algebras}\label{subsection-central.division.algebras}

Let $k$ be any field. Throughout this section, and for the rest of the paper, a $k$-algebra will be a unital, associative algebra $A$ over the field $k$. Its center will be denoted by $\mathrm{Z}(A)$. It always holds that $k \subseteq \mathrm{Z}(A)$. A $k$-algebra $A$ is called \textit{central} if this inclusion is in fact an equality, that is $k = \mathrm{Z}(A)$.

A $k$-algebra $A$ is a \textit{division algebra} if every non-zero element $a \in A$ is two-sided invertible, meaning that there are elements $b,c \in A$ such that $ab = 1$ and $ca = 1$. In this case necessarily $b = c$ and this common element is the unique element satisfying the two equations $ax = 1, ya = 1$. We will denote it by $a^{-1}$, called the \textit{inverse} of $a$. So a division algebra over $k$ is a division ring which is also an algebra over $k$.

In this paper we adopt the following convention.

\begin{definition}\label{definition-central.disivion.algebra}
A $k$-algebra $A$ will be called a \textit{central division algebra} if
\begin{enumerate}[a),leftmargin=0.7cm]
\item $A$ is central;
\item $A$ is a division algebra;
\item $A$ is \textit{finite-dimensional} over $\mathrm{Z}(A) = k$.
\end{enumerate}
We will denote by $[A:k]$ the dimension of $A$ over $k$.
\end{definition}

More generally, if $\calR$ denotes any division ring and if $V$ is a left $\calR$-module, then $[V:\calR]$ will denote the dimension of $V$ over $\calR$.

\subsection{Virtually abelian groups}\label{subsection-virtually.abelian.groups}

A group $\Gamma$ is said to be \textit{virtually abelian} if there exists an abelian subgroup of finite index. Equivalently, if there exists a normal abelian subgroup of finite index.

In this paper we will work with finitely generated groups only. In this case, a finitely generated virtually abelian group $\Gamma$ contains a finitely generated normal abelian subgroup $N$ of finite index. If moreover $\Gamma$ is assumed to be torsion-free, then necessarily $N$ is isomorphic to $\Z^n$ for some $n \geq 1$. 

\subsection{Von Neumann regular and \texorpdfstring{$*$}{}-regular rings}\label{subsection-vN.regular.rings}

A unital ring $R$ is \textit{(von Neumann) regular} if for every element $x \in R$ there exists $y \in R$ such that $xyx = x$. Note that the element $e = xy$ is an idempotent in $R$, that is $e^2 = e$.


A \textit{$*$-regular ring} is a regular ring $R$ endowed with a proper involution $* \colon R \to R$, proper meaning that $x^*x = 0$ implies $x = 0$. It is not difficult to show that in a $*$-regular ring, the following hold. For every element $x \in R$ there exist \textit{unique} projections $e,f \in R$ (that is elements $p \in R$ satisfying $p^2 = p^* = p$) such that $xR = eR$ and $Rx = Rf$, and moreover there exists a unique element $y \in fRe$, called the \textit{relative inverse} of $x$ and denoted by $\overline{x}$, such that $xy = e$ and $yx = f$. We denote by $\text{LP}(x)$ the projection $e$, termed the \textit{left projection} of $x$, and by $\text{RP}(x)$ the projection $f$, termed the \textit{right projection} of $x$.


For any subset $S$ of a $*$-regular ring $R$, there exists a smallest $*$-regular subring of $R$ containing $S$, denoted by $\calR(S,R)$ and termed the $*$-regular closure of $S$ in $R$ \cite[Proposition 6.2]{AraGoodearl}. It always contains the division closure $\calD(S,R)$, which we recall it is defined as the smallest subring of $R$ containing $S$ and closed under taking inverse of elements, when they exist in $R$ (so if $x \in \calD(S,R)$ is invertible in $R$ with inverse $x^{-1}$, then $x^{-1} \in \calD(S,R)$).

\subsection{Crossed products and group algebras}\label{subsection-crossed.products}

Let $R$ be a unital ring and $\Gamma$ be any countable group. A \textit{crossed product}, denoted by $R \ast \Gamma$, is a unital ring in which every element can be uniquely written as a finite sum
$$\sum_{g \in \Gamma} r_g g$$
with $r_g \in R$. The sum is given componentwise, and the product is described by the rule
$$(r g) \cdot (s h) = r \text{ } \sigma_g(s) \text{ } \tau(g,h) \text{ } gh$$
for $r,s \in R$ and $g,h \in \Gamma$, where $\tau : \Gamma \times \Gamma \ra R^{\times}$ and $\sigma : \Gamma \ra \text{Aut}(R)$ are maps satisfying the following properties:
\begin{enumerate}[a),leftmargin=0.7cm]
\item $\sigma_1 = \text{id}_R$, and $\tau(1,h) = \tau(g,1) = 1$ for all $g,h \in \Gamma$;
\item $\tau(g,h) \tau(gh,f) = \sigma_g(\tau(h,f)) \tau(g,hf)$ for all $g,h,f \in \Gamma$ (so $\tau$ is a $2$-cocycle for $\sigma$);
\item $\tau(g,h) \sigma_{gh}(r) = \sigma_g(\sigma_h(r)) \tau(g,h)$ for all $g,h \in \Gamma$, $r \in R$.
\end{enumerate}
Here $R^{\times}$ is the group of units of $R$.

In the particular case that $\tau$ and $\sigma$ are the trivial maps, we obtain the \textit{group ring} $R[\Gamma]$. We will we particularly interested in the case that $R = k$ is a field. If moreover $R = k$ is a field endowed with an involution $*$, it linearly extends to an involution on $k[\Gamma]$ by the rule
$$(\lambda g)^* = \lambda^* g^{-1}$$
for $\lambda \in k, g \in \Gamma$.

\begin{definition}\label{definition-nontrivial.unit}
An element $u \in k[\Gamma]$ is called a \textit{unit} if there exists $v \in k[\Gamma]$ such that $1 = uv = vu$. It will be called a \textit{trivial} unit if $u = \lambda g$ for some $\lambda \in k \backslash \{0\}$ and $g \in \Gamma$, otherwise it is called a \textit{non-trivial} unit.
\end{definition}

Take now $k = \C$ with complex conjugation as involution. Let $l^2(\Gamma)$ denote the Hilbert space with orthonormal basis the elements of $\Gamma$. Then $\C[\Gamma]$ acts faithfully on $l^2(\Gamma)$ by left multiplication, and so we may identify $\C[\Gamma]$ as a subset of $\calB(l^2(\Gamma))$, i.e. as bounded operators on $l^2(\Gamma)$. The \textit{group von Neumann algebra of $\Gamma$}, denoted by $\calN(\Gamma)$, is defined to be the weak closure of $\C[\Gamma]$ in $\calB(l^2(\Gamma))$. Algebraically, it is characterized as follows: $\calN(\Gamma)$ consists of all bounded operators $T \in \calB(l^2(\Gamma))$ which commute with the action of $\C[\Gamma]$ on $l^2(\Gamma)$ given by \textit{right} multiplication.

Also, we denote by $\calU(\Gamma)$ the algebra of (possibly unbounded) operators affiliated to the von Neumann algebra $\calN(\Gamma)$. Again, we can give an algebraic characterization of it: $\calN(\Gamma)$ is an Ore domain, and $\calU(\Gamma)$ coincides with the classical ring of quotients of $\calN(\Gamma)$ \cite[Proposition 2.8]{unboundedop}.

We denote by $\calD(\Gamma) := \calD(\C[\Gamma],\calU(\Gamma))$ the division closure of $\C[\Gamma]$ in $\calU(\Gamma)$, and by $\calR(\Gamma) := \calR(\C[\Gamma],\calU(\Gamma))$ the $*$-regular closure of $\C[\Gamma]$ in $\calU(\Gamma)$. The latter exists since $\calU(\Gamma)$ is $*$-regular \cite[Note 2.11]{unboundedop}. We have the inclusions
$$\C[\Gamma] \subseteq \calD(\Gamma) \subseteq \calR(\Gamma) \subseteq \calU(\Gamma).$$

Note that for $H \leq \Gamma$ a subgroup of $\Gamma$, the inclusion $\C[H] \hookrightarrow \C[\Gamma]$ induces natural embeddings $\calN(H) \hookrightarrow \calN(\Gamma)$, $\calU(H) \hookrightarrow \calU(\Gamma)$ and $\calR(H) \hookrightarrow \calR(\Gamma)$.

\section{Central division algebras and virtually abelian groups}\label{section-virtually.abelian}

Let $\Gamma$ be a finitely generated torsion-free group. In this section we give a necessary and sufficient condition on the $*$-regular closure $\calR(\Gamma)$ in order to guarantee that $\Gamma$ be virtually abelian. The main theorem is the following.

\begin{theorem}\label{theorem-characterization.virtually.abelian}
Let $\Gamma$ be a finitely generated torsion-free group. The following statements are equivalent.
\begin{enumerate}[(1),leftmargin=0.7cm]
\item $\Gamma$ is virtually abelian.
\item $\calR(\Gamma)$ is a central division algebra.
\end{enumerate}
\end{theorem}

The rest of the section is devoted to prove Theorem \ref{theorem-characterization.virtually.abelian} above, although some tangential results will also be obtained.

\subsection{From virtually abelian groups to central division algebras}\label{subsection-virtually.central}

Let us start by taking $\Gamma$ to be a finitely generated, torsion-free, virtually abelian group. Our convention for conjugation is that $h^g := g h g^{-1}$ for $g,h \in \Gamma$.

Let $N \leq \Gamma$ be a maximal normal abelian subgroup of finite index, so $N \cong \Z^n$ for some $n \geq 1$. Let $F := \Gamma/N$, which is finite. We have a natural action of $F$ on $N$ given by conjugation $n \mapsto n^f$, $f \in F$, which linearly extends to an action of $F$ on $k[N]$, also denoted by $p \mapsto p^f$, $f \in F$. Take $\tilde{F}$ to be a fixed right transversal for $N$ i $\Gamma$ and, for simplicity, we take the representative of $N$ to be the unit element $1 \in \Gamma$, that is $\tilde{1} = 1$. If $\pi : \Gamma \ra F$ denotes the natural quotient map, the image of an element $\tilde{f} \in \tilde{F}$ will be simply denoted by $f$, that is $f = \pi(\tilde{f})$. Therefore, for $p \in k[N]$ and $\tilde{f} \in \tilde{F}$,
$$p^{\tilde{f}} = p^f.$$

Fix now any field $k$ endowed with an involution $\ast$. Note that every element of the group algebra $k[\Gamma]$ can be uniquely written as
$$\sum_{\tilde{f} \in \tilde{F}} p_{\tilde{f}} \tilde{f}$$
for some Laurent polynomials $p_{\tilde{f}} \in k[N] \cong k[x_1^{\pm 1},...,x_n^{\pm 1}]$. Therefore, $k[\Gamma]$ is obtained as a crossed product
$$k[\Gamma] = k[N] \ast \tilde{F}$$
whose multiplicative structure is given by
$$(p \tilde{f}) \cdot (q \tilde{g}) = p \text{ } q^f \text{ } \tau(f,g) \text{ } \widetilde{fg},$$
where $\tau : F \times F \ra N$ is the $2$-cocycle given by $\tau(f,g) = \tilde{f} \text{ } \tilde{g} \text{ } \widetilde{fg}^{-1}$. In particular, we see that the map $F \ra \text{Aut}(k[N]), f \mapsto \cdot^f$ is a group homomorphism.

Since $\Gamma$ is virtually abelian, the group ring $k[\Gamma]$ satisfies Kaplansky's zero-divisor conjecture, i.e. it is a domain \cite[Theorem 1.4]{Linnell88}. In fact it is an \textit{Ore domain}. We denote by $\calQ := \calQ(k[\Gamma])$ its Ore field of fractions, which is a division algebra over $k$.

The following notation will be used throughout the section.

\begin{notation}\label{notation-calQ}
\text{ }
\begin{enumerate}[a),leftmargin=0.7cm]
\item We will denote by $\calZ$ the center of $\calQ$, which is a field. 
\item For a subset $A \subseteq \calQ$, we denote by $\calZ(A)$ the minimal division algebra over $k$ containing $\calZ$ and the elements of $A$.
\item For any division ring $\calR$ and any $\alpha,\beta \in \calR$ with $\beta \neq 0$, conjugation of $\alpha$ by $\beta$ will be denoted by $\alpha^{\beta} := \beta \alpha \beta^{-1}$.
\end{enumerate}
\end{notation}

We start by computing, in the case $k = \C$, the division closure $\calD(\Gamma)$ and the $*$-regular closure $\calR(\Gamma)$ in terms of $\calQ$. It turns out that these three rings coincide.

\begin{proposition}\label{proposition-computation.closures}
For $k = \C$, we have
$$\calQ = \calD(\Gamma) = \calR(\Gamma).$$
\end{proposition}
\begin{proof}
The inclusions $\calQ \subseteq \calD(\Gamma) \subseteq \calR(\Gamma)$ are obvious. Now, $\calQ$ is a division algebra, so it is von Neumann regular. It is also $*$-closed and contains $k[\Gamma]$, so $\calR(\Gamma) \subseteq \calQ$, finishing the proof of the proposition.
\end{proof}

Let us return to the general setting where $k$ is an arbitrary field with involution $\ast$. We define $\calL$ to be the field of fractions of $k[N] \cong k[x_1^{\pm 1},...,x_n^{\pm 1}]$, which can be seen as a subfield of $\calQ$. In fact, it can be identified with $k(x_1,...,x_n)$, the field of rational functions in $n$ variables $x_1,...,x_n$, and we will do so throughout the paper. Each element $f \in F$ acts on $\calL$ by conjugation, so it naturally induces an element of $\text{Autl}(\calL/k)$, which we will denote by $\alpha_f$. It is easy to see, using maximality of the abelian subgroup $N$, that the map $\alpha \colon F \ra \text{Aut}(\calL/k), f \mapsto \alpha_f$ is an injective group homomorphism, so we identify $F$ with the subgroup $H := \text{im}(\alpha)$.

\begin{proposition}\label{proposition-galois}
The following statements hold true.
\begin{enumerate}[(i),leftmargin=0.7cm]
\item The ring $\calQ$ is Artinian, and moreover
$$\calQ = \calL \ast \tilde{F}$$
for a suitable crossed product. In particular, $\tilde{F}$ is a basis of $\calQ$ over $\calL$, where $\calL$-multiplication is given from the left. Therefore $[\calQ : \calL] = |F|$.
\item If we denote by $\calL^H$ the fixed subfield of $\calL$ associated with the subgroup $H \leq \emph{Aut}(\calL/k)$, then
$$\calZ = \calL^H.$$
In particular, $\calQ$ is finite-dimensional over $\calZ$, and in fact $[\calQ : \calZ] = |F|^2$.
\end{enumerate}
\end{proposition}
\begin{proof}
\begin{enumerate}[(i),leftmargin=0.7cm]
\item Denote by $\calA$ the subring of $\calQ$ generated by $\calL$ and $\Gamma$. Every element of $\calA$ is a finite sum of monomials
$$x_1x_2 \cdots x_r$$
where either $x_i \in \calL$ or $x_i \in \tilde{F}$. Note that if $x \in \calL$ and $g = n \tilde{f} \in \Gamma$ for unique elements $n \in N, \tilde{f} \in \tilde{F}$, then
$$gx = n\tilde{f}x = n \tilde{f}x\tilde{f}^{-1} \tilde{f} = n x^{\tilde{f}} \tilde{f} \in \sum_{\tilde{f} \in \tilde{F}} \calL \tilde{f}.$$
With this we deduce that $\calA = \sum_{\tilde{f} \in \tilde{F}} \calL \tilde{f}$. To prove that this sum is direct, suppose that we have
\begin{equation}\label{equation-crossed.product}
\sum_{\tilde{f} \in \tilde{F}} x_{\tilde{f}} \tilde{f} = \sum_{\tilde{f} \in \tilde{F}} y_{\tilde{f}} \tilde{f}
\end{equation}
for some $x_{\tilde{f}}, y_{\tilde{f}} \in \calL$. Write
$$x_{\tilde{f}} = p_{\tilde{f}} q_{\tilde{f}}^{-1}, \quad y_{\tilde{f}} = r_{\tilde{f}} s_{\tilde{f}}^{-1}$$
where each $p_{\tilde{f}}, q_{\tilde{f}}, r_{\tilde{f}}, s_{\tilde{f}} \in k[N]$, $q_{\tilde{f}}, s_{\tilde{f}} \neq 0$. By multiplying \eqref{equation-crossed.product} on the left with the element $h = \prod_{\tilde{f} \in \tilde{F}} q_{\tilde{f}} s_{\tilde{f}} \neq 0$ we get
$$\sum_{\tilde{f} \in \tilde{F}} x_{\tilde{f}}h \tilde{f} = \sum_{\tilde{f} \in \tilde{F}} y_{\tilde{f}} h \tilde{f},$$
which is an equality in $k[\Gamma]$. But $k[\Gamma] = k[N] \ast \tilde{F}$, so necessarily $x_{\tilde{f}}h = y_{\tilde{f}}h$ for every $\tilde{f} \in \tilde{F}$, which implies $x_{\tilde{f}} = y_{\tilde{f}}$ as $h \neq 0$. This proves that the sum $\sum_{\tilde{f} \in \tilde{F}} \calL \tilde{f}$ is direct, and so $\calA = \calL \ast \tilde{F}$ for a suitable crossed product. In particular, $\calA$ is Artinian, and we deduce that $\calA = \calQ$. Indeed, given $ab^{-1} \in \calQ$ with $a,b \in k[\Gamma]$, $b \neq 0$, then the chain of ideals
$$b\calA \supseteq b^2 \calQ \supseteq \cdots \supseteq b^n \calA \supseteq \cdots$$
stabilizes, so we can find a positive integer $k \geq 1$ and $c \in \calA$ such that $b^k = b^{k+1}c$. This implies that $b^{-1} = c \in \calA$, and so $ab^{-1} \in \calA$, proving the equality $\calQ = \calA = \calL \ast \tilde{F}$.

The second part now easily follows.
		
\item Let us first show the inclusion $\calL^H \subseteq \calZ$. Given $x \in \calL^H$, it is clear that $x$ commutes with any $n \in N$. Take now $\tilde{f} \in \tilde{F}$. We compute
\begin{align*}
\tilde{f} x \tilde{f}^{-1} & = x^{\tilde{f}} = x^f = \alpha_f(x) = x, 
\end{align*}
so $x$ also commutes with $\tilde{F}$. Therefore $x$ is central in $k[\Gamma]$, and so also central in $\calQ$.
		
For the other inclusion $\calZ \subseteq \calL^H$, we first note the following. Due to maximality of $N \trianglelefteq \Gamma$, we have that $\ker(N \xrightarrow{\cdot^f} N) \neq N$ for any $f \in F \backslash \{1\}$. This says that, for any $\tilde{f} \in \tilde{F} \backslash \{1\}$, there exists $n \in N$ such that $n^f \neq n$.
		
Take $x \in \calZ$. Using part (i) we can uniquely write it as
$$x = \sum_{\tilde{f} \in \tilde{F}} l_{\tilde{f}} \tilde{f}$$
for some $l_{\tilde{f}} \in \calL$.  Fix $\tilde{g} \in \tilde{F} \backslash \{1\}$, and consider $n \in N$ such that $n^g \neq n$. Then $x$ commutes with $n$, so we have the equality
$$\sum_{\tilde{f} \in \tilde{F}} l_{\tilde{f}} n^f \tilde{f} = \sum_{\tilde{f} \in \tilde{F}} l_{\tilde{f}} n \tilde{f}.$$
By uniqueness, necessarily
$$l_{\tilde{f}} (n^f - n) = 0$$
for all $\tilde{f} \in \tilde{F}$. In particular, $l_{\tilde{g}} = 0$. We thus conclude that $x = l_{\tilde{1}} \tilde{1} = l_1 \in \calL$. Now, given any $f \in F$, we compute
$$\alpha_f(x) = \tilde{f} x \tilde{f}^{-1} = x$$
as $x \in \calZ$. This says that $x \in \calL^H$, and hence establishes the other inclusion. The last statement now follows from part (i) and the fact that $\calL / \calL^H$ is a finite Galois extension with Galois group $H$.
\end{enumerate}
\end{proof}

As a consequence of part (ii) of the above proposition, we see that $\calQ$ is finite-dimensional over its center $\calZ$. This turns $\calQ$ into a central division algebra when considered as an algebra over $\calZ$. Now one of the implications of Theorem \ref{theorem-characterization.virtually.abelian} is immediate.

\begin{proof}[Proof of $(1) \implies (2)$ of Theorem \ref{theorem-characterization.virtually.abelian}]
Just use Proposition \ref{proposition-computation.closures} and the conclusions above.
\end{proof}

\subsection{From central division algebras to virtually abelian groups}\label{subsection-central.virtually}

We start the section by stating the following theorem.

\begin{theorem}\label{theorem-finite.virtually.abelian}
Let $\Gamma$ be a finitely generated torsion-free group with $\calR(\Gamma)$ being a semisimple ring. Suppose that $\calR(\Gamma)$ is finitely generated as a $\emph{Z}(\calR(\Gamma))$-module. Then $\Gamma$ is virtually abelian.
\end{theorem}

Before proving it, we need a couple of technical lemmas.

\begin{lemma}\label{lemma-preliminary.vNregular}
Let $G$ be a group with $\calR(G)$ being a semisimple ring. Let $H \leq G$ be a subgroup, and define $\calK := \calR(G) \cap \calU(H)$. Then $\calK$ is a regular ring.

Moreover, $\calR(G)$ becomes a $\calK$-module, and if $\calR(G)$ is finitely generated as a $\calK$-module, then $[G:H] < +\infty.$
\end{lemma}
\begin{proof}
First, note that $\calK$ can be considered as a subring of $\calU(G)$ because of the embedding $\calU(H) \hookrightarrow \calU(G)$. Also, note that all three rings $\calR(G), \calU(H)$ and $\calU(G)$ are $*$-regular.

Let us first prove that $\calK$ is regular. Let $x \in \calK$ be given. Since $x \in \calR(G) \subseteq \calU(G)$ and $x \in \calU(H)$, there exist unique projections $e_1,f_1 \in \calR(G)$, $e_2,f_2 \in \calU(G)$ and $e_3,f_3 \in \calU(H)$ such that
$$x \calR(G) = e_1 \calR(G), \quad x \calU(G) = e_2 \calU(G), \quad x \calU(H) = e_3 \calU(H)$$
and
$$\calR(G) x = \calR(G) f_1, \quad \calU(G) x = \calU(G) f_2, \quad \calU(H) x = \calU(H) f_3,$$
and moreover there exist unique elements $\overline{x}_1 \in \calR(G), \overline{x}_2 \in \calU(G)$ and $\overline{x}_3 \in \calU(H)$ such that $x \overline{x}_i = e_i$ and $\overline{x}_i x = f_i$ for all $i=1,2,3$. In particular, note that $e_ix = x = x f_i$ for all $i=1,2,3$. But now we have
$$e_1 \calU(G) = x \overline{x}_1 \calU(G) \subseteq x \calU(G) = e_1 x \calU(G) \subseteq e_1 \calU(G),$$
$$e_3 \calU(G) = x \overline{x}_3 \calU(G) \subseteq x \calU(G) = e_3 x \calU(G) \subseteq e_3 \calU(G),$$
and so $e_1 \calU(G) = e_3 \calU(G) = x \calU(G) = e_2 \calU(G)$.By uniqueness, $e := e_1 = e_2 = e_3$. Similarly, $f := f_1 = f_2 = f_3$. But now by uniqueness of the elements $\overline{x}_i$ we also conclude that $\overline{x} := \overline{x}_1 = \overline{x}_2 = \overline{x}_3 \in \calR(G) \cap \calU(H) = \calK$. This proves that $\calK$ is a regular ring, since
$$x \overline{x} x = e x = x.$$

It is clear that $\calR(G)$ is a (right) $\calK$-module. Suppose now that $\calR(G)$ is generated by $n$ elements as a $\calK$-module, say by $\{r_1,...,r_n\} \subseteq \calR(G)$. This gives us a surjective $\calK$-module map
$$\alpha \colon \calK^n \twoheadrightarrow \calR(G), \quad (u_1,...,u_n) \mapsto \sum_{i=1}^n r_i u_i.$$
Suppose, by way of contradiction, that $[G:H] = +\infty$, and take $S = \{s_1,...,s_{n+1}\}$ a set of transversals of $H$ consisting of $n+1$ elements. This set gives us a $\calK$-module map
$$\beta \colon \calK^{n+1} \ra \calR(G), \quad (u_1,...,u_{n+1}) \mapsto \sum_{i=1}^{n+1} s_i u_i.$$
But since free modules are projective, we can lift $\beta$ to a $\calK$-module map $\gamma \colon \calK^{n+1} \to \calK$ such that

\[
\begin{tikzcd}
& \calK^{n+1} \arrow[dl,dotted,swap,"\gamma"] \arrow[d,"\beta"] \\
\calK^n \arrow[r,two heads,"\alpha"] & \calR(G) 
\end{tikzcd}
\]

If it were the case that $\beta$ is injective, then so would be $\gamma$. But regular rings are absolutely flat, so tensoring $\gamma$ with $\text{id}_{\calR(G)}$ would yield an injective $\calR(G)$-module map
$$\gamma \otimes \text{id}_{\calR(G)} \colon \calR(G)^{n+1} \to \calR(G)^n.$$
This is impossible as $\calR(G)$ is a semisimple ring, hence Artinian.

We conclude that $\beta$ is not injective, which means that there exist elements $u_1,...,u_{n+1} \in \calK$ not all zero such that
$$u := \sum_{i=1}^{n+1} s_i u_i = 0.$$
We now think of each $u_i \in \calK \subseteq \calU(H)$ as a densely defined closed operator $u_i \colon D_i \to l^2(H)$ commuting with the right action of $H$, with $D_i \subseteq l^2(H)$ dense in $l^2(H)$ and invariant under the right action of $H$. The operator $u \colon D \to l^2(H)$ is defined on $D := \bigcap_{i=1}^{n+1} D_i$, which is dense in $l^2(H)$ because each $D_i$ is essentially dense (see \cite[Lemma 8.3]{Luck}). Now for each $l \in D$, we compute
$$0 = u(l) = \sum_{i=1}^{n+1} s_i u_i(l),$$
with each $s_i u_i(l) \in s_i l^2(H) \subseteq l^2(G)$. But since $S = \{s_1,...,s_{n+1}\}$ is a set of transversals of $H$ in $G$, the Hilbert subspaces $s_1 l^2(H),...,s_{n+1} l^2(H)$ are pairwise orthogonal. 
We thus deduce that each $s_i u_i(l) = 0$, that is $u_i(l) = 0$. We conclude that each operator $u_i$ is zero on the dense subspace $D$. Since each $u_i$ is closed, this implies that $u_i \equiv 0$. This contradicts the fact that at least one of the elements $u_1,...,u_{n+1}$ is non-zero. Therefore necessarily $[G:H] < +\infty$, and in fact we have proved the more stronger result $[G:H] \leq n$.
\end{proof}

\begin{lemma}\label{lemma-group.central}
Let $G$ be a group and consider $\Delta(G)$ to be the subgroup of $G$ consisting of all elements whose conjugacy class in $G$ is finite. Then
$$\emph{Z}(\calU(G)) \subseteq \calU(\Delta(G)).$$
\end{lemma}
\begin{proof}
Take $x \in \text{Z}(\calU(G))$, so we can write it as $x = ab^{-1}$ for some $a,b \in \calN(G)$. In fact, by using the polar decomposition and the spectral theorem for elements of $\calU(G)$, we can even take $a,b$ to be in the center of $\calN(G)$. Thus in order to prove the lemma it is enough to show the inclusion $\text{Z}(\calN(G)) \subseteq \calN(\Delta(G))$. For this, we will use the identification of $\calN(G)$ as a subspace of $l^2(G)$ given by
$$\calN(G) \hookrightarrow l^2(G), \quad T \mapsto T(1).$$
So take $T \in \text{Z}(\calN(G))$ and write it as $T = \sum_{g \in G}\lambda_g g \in l^2(G)$. In particular $T$ commutes with every $h \in G$, and so we deduce that $\lambda_{hgh^{-1}} = \lambda_g$ for all $g,h \in G$. So if $g$ has infinite conjugacy class in $G$, the $l^2$-summability of $T$ implies that necessarily $\lambda_g = 0$. This shows that $T$ is supported on elements in $\Delta(G)$, which is what we wanted to prove.
\end{proof}

We are now ready to prove Theorem \ref{theorem-finite.virtually.abelian}.

\begin{proof}[Proof of Theorem \ref{theorem-finite.virtually.abelian}]
Write $\calZ := \text{Z}(\calR(\Gamma))$ for the center of $\calR(\Gamma)$. Let $\Delta(\Gamma)$ be the subgroup of $\Gamma$ consisting of all elements $g \in G$ whose conjugacy class in $\Gamma$ is finite. Consider
$$\calK := \calR(\Gamma) \cap \calU(\Delta(\Gamma)).$$
We claim that $\calZ \subseteq \calK$. For this, let $z \in \calZ = \text{Z}(\calR(\Gamma))$. In particular, $z$ commutes with every $g \in \Gamma$, and hence it commutes with every element of $\calN(\Gamma)$. Since $\calU(\Gamma)$ is the classical ring of quotients of $\calN(\Gamma)$, we conclude that $z$ also commutes with every element of $\calU(\Gamma)$. This shows that
$$\calZ \subseteq \calR(\Gamma) \cap \text{Z}(\calU(\Gamma)).$$
The claim follows due to Lemma \ref{lemma-group.central}.

Now since $\calR(\Gamma)$ is finitely generated as a $\calZ$-module, it is also finitely generated as a $\calK$-module since $\calZ \subseteq \calK$. Thus Lemma \ref{lemma-preliminary.vNregular} tells us that $[\Gamma:\Delta(\Gamma)] < +\infty$. In particular, since $\Gamma$ is finitely generated and $\Delta(\Gamma)$ is a subgroup of finite index, then so is $\Delta(\Gamma)$. Pick $S = \{g_1,..,g_r\}$ to be a finite generating subset of $\Delta(\Gamma)$, and consider the centralizer of $S$ in $\Delta(\Gamma)$, that is
$$C_{\Delta(\Gamma)}(S) = \{h \in \Delta(\Gamma) \mid hgh^{-1} = g \text{ for all } g \in S\}.$$
Note that
$$C_{\Delta(\Gamma)}(S) = \bigcap_{i=1}^r C_{\Delta(\Gamma)}(\{g_i\}).$$
For each $g_i \in S$, since its conjugacy class is finite, say with $n_i \in \N$ elements, we can pick $h_{i1},...,h_{in_i} \in \Gamma$ such that the conjugacy class of $g_i$ consists of the elements
$$\{h_{i1}g_ih_{i1}^{-1},...,h_{in_i}g_ih_{in_i}^{-1}\}.$$
In this case, the set $J_i := \{h_{i1},...,h_{in_i}\}$ contains a transversal of $C_{\Delta(\Gamma)}(\{g_i\})$ in $\Delta(\Gamma)$. 
In particular, $[\Delta(\Gamma):C_{\Delta(\Gamma)}(\{g_i\})] < +\infty$. Since the intersection of finite index subgroups is a finite index subgroup, we conclude that
$$[\Delta(\Gamma):C_{\Delta(\Gamma)}(S)] < +\infty.$$
But being $S$ a generating set for $\Delta(\Gamma)$ implies that the subgroup $C_{\Delta(\Gamma)}(S)$ is central, and therefore abelian. Finally,$$[\Gamma:C_{\Delta(\Gamma)}(S)] = [\Gamma:\Delta(\Gamma)] \cdot [\Delta(\Gamma):C_{\Delta(\Gamma)}(S)] < +\infty,$$
and so $C_{\Delta(\Gamma)}(S)$ is an abelian subgroup of $\Gamma$ with finite index. This proves precisely that $\Gamma$ is virtually abelian, as required.
\end{proof}

We can finally prove the other implication of Theorem \ref{theorem-characterization.virtually.abelian}.

\begin{proof}[Proof of $2) \implies 1)$ of Theorem \ref{theorem-characterization.virtually.abelian}]
This is a special case of Theorem \ref{theorem-finite.virtually.abelian}.
\end{proof}

\section{On the existence of units in group rings of virtually abelian groups}\label{section-general.procedure}

In this section we present a general procedure which can help in the understanding of units in group rings of virtually abelian groups. This procedure will be applied to the case of the Promislow's group in Section \ref{section-criterion.promislow}. We will follow the same notation as in Section \ref{subsection-virtually.central}. Since we have already showed that $\calQ$ is a central division algebra, the following result from the theory of central division algebras will prove useful for us in the sequel.

\begin{theorem}\label{theorem-central.division.algebras}
Let $\calR$ be a division algebra, finite-dimensional over its center $\emph{Z}(\calR)$. The following statements hold.
\begin{enumerate}[(1),leftmargin=0.7cm]
\item $[\calR : \emph{Z}(\calR)] = n^2$ for some $n \in \N$. Furthermore, if $\calK \subseteq \calR$ is a maximal subfield, then $[\calK : \emph{Z}(\calR)] = n$.
\item (Skolem-Noether) Let $\calA_1,\calA_2 \subseteq \calR$ be two isomorphic $\emph{Z}(\calR)$-subalgebras. Then for any $\emph{Z}(\calR)$-algebra homomorphism $\phi \colon \calA_1 \to \calA_2$ there is an element $b \in \calR$ such that $\phi$ is given by conjugation by $b$. In symbols,
$$\phi(a) = a^b$$
for all $a \in \calA_1$.
\end{enumerate}
\end{theorem}
\begin{proof}
Part (1) follows from \cite[Theorem 3.10 and Corollary 3.17]{CentralAlgebras}. Part (2) is a special case of \cite[Theorem 3.14]{CentralAlgebras}.
\end{proof}

An immediate observation, due to Theorem \ref{theorem-central.division.algebras} (1) and Proposition \ref{proposition-galois} (ii), is that $\calL \subseteq \calQ$ is a maximal subfield of $\calQ$ which contains its center $\calZ = \calL^H$.

We now expose the key idea on constructing units in $k[\Gamma]$. Let $p(x) \in \text{Z}(k[\Gamma])[x]$ be a polynomial with coefficients in the ring $\text{Z}(k[\Gamma]) \subseteq \calZ$, and assume that $1-p(x)$ is \textit{not} irreducible over $\text{Z}(k[\Gamma])$, so we can write
$$1-p(x) = q_1(x) q_2(x)$$
for some non-trivial polynomials $q_1(x),q_2(x) \in \text{Z}(k[\Gamma])[x]$. We then observe that, if $T \in k[\Gamma]$ is a root of the polynomial $p(x)$, then
$$1 = q_1(T) q_2(T),$$
so $q_1(T)$ is a unit in $k[\Gamma]$, which can be either trivial or non-trivial.

We will concentrate in polynomials of the form
$$P_{\kappa,d}(x) = x^d - \kappa^d + 1$$
for some $\kappa \in \text{Z}(k[\Gamma])$ and $d \in \N$, for which in this case
$$1 - P_{\kappa,d}(x) = \kappa^d -x^d = (\kappa - x)(x^{d-1} + \kappa x^{d-2} + \cdots + \kappa^{d-2} x + \kappa^{d-1}).$$

\begin{definition}\label{definition-conjugate}
Let $\omega \in k[\Gamma]$ be a root of $P_{\kappa,d}(x)$. An element $T \in \calQ$ is said to be an \emph{$\omega$-conjugate of degree $d$} if $T^d = \omega^d$. We call such a $T$ \emph{trivial} if it is of the form $\xi \omega$ where $\xi \in k$ is a $d^{\text{th}}$ root of unity. If $T$ is not trivial, we say that it is \emph{non-trivial}.
\end{definition}

For $\omega \neq 0$, Theorem \ref{theorem-central.division.algebras} (2) implies that every $\omega$-conjugate is conjugate, in $\calQ$, to $\omega$. Hence the above definition is consistent.

The following proposition summarizes the previous idea on finding units in $k[\Gamma]$.

\begin{proposition}\label{proposition-general.criterion}
Let $\omega \in k[\Gamma]$ be a non-zero root of $P_{\kappa,d}(x)$, and let $T \in \calQ$ be an $\omega$-conjugate of degree $d$.
\begin{enumerate}[(i),leftmargin=0.7cm]
\item The element $U_T := \kappa - T$ is a unit in $\calQ$, with inverse $U_T^{-1} = T^{d-1} + \kappa T^{d-2} + \cdots + \kappa^{d-2} T + \kappa^{d-1}$.
\item $T \in k[\Gamma]$ if and only if $U_T \in k[\Gamma]$, in which case $U_T^{-1}$ also belongs in $k[\Gamma]$.
\end{enumerate}
Conversely, let $U \in \calQ \backslash \{0\}$, and define $T_U := \omega^U = U \omega U^{-1}$.
\begin{enumerate}[(i),leftmargin=0.7cm]
\item The element $T_U$ is an $\omega$-conjugate of degree $d$.
\item If moreover $U,U^{-1} \in k[\Gamma]$, then $T_U \in k[\Gamma]$.
\end{enumerate}
\end{proposition}
\begin{proof}
The first part is a simple computation:
$$U_T U_T^{-1} = (\kappa - T)(T^{d-1} + \kappa T^{d-2} + \cdots + \kappa^{d-2} T + \kappa^{d-1}) = \kappa^d - T^d = \kappa^d - \omega^d = 1.$$
Similarly $U_T^{-1} U_T = 1$. Moreover, from the definition $U_T = \kappa - T$ it is clear that $T \in k[\Gamma]$ if and only if $U_T \in k[\Gamma]$, in which case
$$U_T^{-1} = T^{d-1} + \kappa T^{d-2} + \cdots + \kappa^{d-2} T + \kappa^{d-1} \in k[\Gamma]$$
as well.

Conversely, let $U \in \calQ \backslash \{0\}$ and define $T_U := \omega^U$. We compute
$$T_U^d = U \omega^d U^{-1} = \omega^d$$
as $\omega^d = \kappa^d - 1 \in \calZ$. Hence $T_U$ is an $\omega$-conjugate of degree $d$. Since $\omega \in k[\Gamma]$, it is clear that if $U,U^{-1} \in k[\Gamma]$ then $T_U \in k[\Gamma]$ as well.
\end{proof}

Proposition \ref{proposition-general.criterion} does not give conditions about the non-triviality of the units $U_T,U_T^{-1}$ in case they belong to $k[\Gamma]$. Nevertheless, we will show in the next section that, in case $\Gamma$ is the Promislow's group, the non-triviality of $U_T, U_T^{-1} \in k[\Gamma]$ is equivalent to the non-triviality of the $\omega$-conjugate $T$, and analogously for the non-triviality of $T_U$ in terms of the non-triviality of $U,U^{-1} \in k[\Gamma]$.

\section{The case of the Promislow's group}\label{section-criterion.promislow}

For a group $\Gamma$, it will be convenient for us to denote by $\ol{g}$ the inverse of $g \in \Gamma$.

In this section we will apply the general procedure displayed in Section \ref{section-general.procedure} to the case of the Promislow's group. This group, which will be denoted by $G$, is given by the presentation
$$G = \langle s,t \mid t s^2 \ol{t} = \ol{s}^2, s t^2 \ol{s} = \ol{t}^2 \rangle.$$
We define $u := st$ together with $x := s^2, y := t^2$ and $z := u^2$. We can then write the above presentation of the group as
$$G = \langle s,t,x,y \mid x^t = \ol{x}, y^s = \ol{y}, x=s^2, y=t^2 \rangle.$$
Using this, it is clear that $G$ is given by the amalgam $G_1 \ast_H G_2$, with $G_1 = \langle t,x \mid x^t = \ol{x} \rangle$, $G_2 = \langle s,y \mid y^s = \ol{y} \rangle$ and $H = \langle x,t^2 \rangle = \langle s^2,y \rangle$. We note that $H \cong \Z^2$ and $G_1 \cong G_2 \cong \Z \rtimes \Z$, where the automorphism implementing the crossed product is the inversion automorphism. This shows that $G$ is a torsion-free group.

Another important property of $G$ is that it fits in a non-split short exact sequence
$$1 \ra \Z^3 \ra G \ra \Z/2 \Z \times \Z/2 \Z \ra 1,$$
where $\Z^3$ is identified with the maximal abelian normal subgroup $N = \langle x,y,z \rangle$, and $F = G/N =\{N,Ns,Nt,Nu\} \cong \Z/2\Z \times \Z/2 \Z$. In particular, $G$ is a virtually abelian group. The natural action of $F$ on $N$ by conjugation is explicitly described in Table \ref{table-conjugation.group} below.

\begin{table}[h]
\centering
\begin{tabular}{c|c|c|c|}
\cline{2-4}
conj & $x$ & $y$ & $z$ \\
\hline
\hline
\multicolumn{1}{|c||}{$N$} & $x$ & $y$ & $z$ \\
\hline
\multicolumn{1}{|c||}{$Ns$} & $x$ & $\ol{y}$ & $\ol{z}$ \\
\hline
\multicolumn{1}{|c||}{$Nt$} & $\ol{x}$ & $y$ & $\ol{z}$ \\
\hline
\multicolumn{1}{|c||}{$Nu$} & $\ol{x}$ & $\ol{y}$ & $z$ \\
\hline
\end{tabular}
\caption{The first column elements act on the first row elements by conjugation.}
\label{table-conjugation.group}
\end{table}

We will fix the right transversal of $N$ in $G$ to be $\tilde{F} = \{1,s,t,u\}$. The associated $2$-cocycle is explicitly written down in Table \ref{table-2.cocycle}.

\begin{table}[h]
\centering
\begin{tabular}{c|c|c|c|c|}
\cline{2-5}
$\tau(a,b)$ & $1$ & $s$ & $t$ & $u$ \\
\hline
\hline
\multicolumn{1}{|c||}{$1$} & $1$ & $1$ & $1$ & $1$ \\
\hline
\multicolumn{1}{|c||}{$s$} & $1$ & $x$ & $1$ & $x$ \\
\hline
\multicolumn{1}{|c||}{$t$} & $1$ & $\bar{x} y \bar{z}$ & $y$ & $\bar{x} \bar{z}$ \\
\hline
\multicolumn{1}{|c||}{$u$} & $1$ & $\bar{y} z$ & $\bar{y}$ & $z$ \\
\hline
\end{tabular}
\caption{The first column corresponds to the different values of $a$, and the first row corresponds to the different values of $b$.}
\label{table-2.cocycle}
\end{table}

For the following discussion, we take $k$ to be any field with involution $*$ with characteristic different from $2$. For $g \in \{x,y,z\}$, we let
$$\kappa_g := \frac{1}{2}(g + \ol{g}), \text{ } \omega_g := \frac{1}{2}(g - \ol{g}) \in k[x^{\pm 1},y^{\pm 1},z^{\pm 1}].$$
We note that
\begin{equation}\label{equation-omega.kappa}
\kappa_g^2-\omega_g^2 = 1.
\end{equation}

Following Sections \ref{subsection-virtually.central} and \ref{section-general.procedure}, we denote by $\calQ$ the Ore field of fractions of $k[G]$, by $\calL$ the Ore field of fractions of $k[N]$, namely $\calL = k(x,y,z)$, and the group $F$ is identified, as a subgroup of $\text{Aut}(\calL/k)$, with $H = \{\text{id},\alpha_s,\alpha_t,\alpha_u\}$.
Apart from the automorphisms in $H$, we have three natural automorphisms of $\calL$, namely
$$\iota_x : \calL \ra \calL, \quad \iota_x(x) = \ol{x}, \iota_x(y) = y, \iota_x(z) = z;$$
$$\iota_y : \calL \ra \calL, \quad \iota_y(x) = x, \iota_y(y) = \ol{y}, \iota_y(z) = z;$$
$$\iota_z : \calL \ra \calL, \quad \iota_z(x) = x, \iota_z(y) = y, \iota_z(z) = \ol{z}.$$
That is, for $g \in \{x,y,z\}$, $\iota_g$ is the automorphism induced by the inversion map with respect to $g$. The next proposition is an extension of Proposition \ref{proposition-galois} in the particular case of the Promislow's group.

\begin{proposition}\label{proposition-galois.promislow}
The following statements hold true.
\begin{enumerate}[(i),leftmargin=0.7cm]
\item The set $\{1,s,t,u\}$ is a basis of $\calQ$ over $\calL$, where $\calL$-multiplication is given from the left. Therefore $[\calQ : \calL] = 4$.
\item When writing an element $\alpha \in \calQ$ as $\alpha = a + bs + ct + du$ with $a,b,c,d \in \calL$, we get the formulas
\begin{align*}
& a = \frac{1}{4}(\alpha + \alpha^{\omega_x} + \alpha^{\omega_y} + \alpha^{\omega_z}), \qquad bs = \frac{1}{4}(\alpha + \alpha^{\omega_x} - \alpha^{\omega_y} - \alpha^{\omega_z}), \\
& ct = \frac{1}{4}(\alpha - \alpha^{\omega_x} + \alpha^{\omega_y} - \alpha^{\omega_z}), \qquad du = \frac{1}{4}(\alpha - \alpha^{\omega_x} - \alpha^{\omega_y} + \alpha^{\omega_z}).
\end{align*}
\item The fixed field of the subgroup $H$ is exactly $k(\kappa_x,\kappa_y,\kappa_z,\xi)$, where $\xi := \omega_x\omega_y\omega_z$, which is an extension of $\calK := k(\kappa_x,\kappa_y,\kappa_z)$ of order $2$. In particular, $\calZ = \calK(\xi)$.
\item We have $[\calQ : \calZ] = 16$, and $\calL$ is a maximal subfield of $\calQ$ which contains $\calZ$.
\end{enumerate}
\end{proposition}
\begin{proof}
\begin{enumerate}[(i),leftmargin=1cm]
\item This is part (i) of Proposition \ref{proposition-galois}.

\item We conjugate $\alpha = a + bs + ct + du$ with the elements $\omega_x,\omega_y,\omega_z$ to get
$$\alpha^{\omega_x} = a + bs - ct - du,$$
$$\alpha^{\omega_y} = a - bs + ct - du,$$
$$\alpha^{\omega_z} = a - bs - ct + du.$$
From here we obtain the displayed formulas.

\item Clearly the elements $\kappa_x,\kappa_y,\kappa_z$ and $\xi = \omega_x \omega_y \omega_z$ are fixed by $H$, so $\calK(\xi) = k(\kappa_x,\kappa_y,\kappa_z,\xi) \subseteq \calL^H$.

Consider now the field extensions
$$\calK \subseteq \calK(\omega_x) \subseteq \calK(\omega_x,\omega_y) \subseteq \calK(\omega_x,\omega_y,\omega_z) = \calL.$$
Since $\omega_g^2 = \kappa_g^2 - 1 \in \calK$ for any $g \in \{x,y,z\}$, all three inclusions are of degree at most $2$. We prove that the first one is exactly $2$ by showing that $\omega_x \notin \calK$. But this is immediate, as $\iota_x(\omega_x) = -\omega_x$, and $\iota_x$ restricts to the identity over $\calK$. The other inclusions are analogously proven to be strict by using the other automorphisms $\{\iota_y,\iota_z\}$. Thus $[\calL:\calK] = 8$. On the other hand, we observe that $\calK(\xi)$ is an extension of $\calK$ of order $2$, since $\xi^2 = (\kappa_x^2-1)(\kappa_y^2-1)(\kappa_z^2-1) \in \calK$ but, for example, $\iota_x(\xi) = -\xi$, so $\xi \notin \calK$.

We now deduce that $[\calL:\calK(\xi)] = 4$. But $[\calL:\calL^H] = |H| = 4$, hence $\calL^H = \calK(\xi)$. The statement now follows from part (ii) of Proposition \ref{proposition-galois}, which tells us that $\calZ = \calL^H$.

\item This is part (ii) of Proposition \ref{proposition-galois}, together with Theorem \ref{theorem-central.division.algebras} (1).
\end{enumerate}
\end{proof}
For $g,h,j \in \{x,y,z\}$ we will denote by $\iota_{gh}$ the composition $\iota_g\iota_h$, and by $\iota_{ghj}$ the composition $\iota_g\iota_h\iota_j$. We note that conjugation by $s$ is equal to $\iota_{yz}$ on $\calL$, i.e. $\alpha_s = \iota_{yz}$. In fact, $\alpha_t = \iota_{xz}$ and $\alpha_u = \iota_{xy}$, so $H = \{\text{id},\iota_{xy},\iota_{xz},\iota_{yz}\}$ (see Table \ref{table-conjugation.group}).\\

\begin{corollary}\label{corollary-galois.promislow}
The set $\{1,\omega_x,\omega_y,\omega_z\}$ is a $\calZ$-basis for $\calL$.
\end{corollary}
\begin{proof}
Since $\calZ = \calL^H$ by Proposition \ref{proposition-galois}, we have $[\calL:\calZ] = [\calL:\calL^H] = |H| = 4$ and $\text{Gal}(\calL/\calZ) = H = \{\text{id},\iota_{xy},\iota_{xz},\iota_{yz}\}$. Now, linear independence of the set $\{1,\omega_x,\omega_y,\omega_z\}$ follows by applying the Galois automorphisms to a $\calZ$-linear combination $a + b\omega_x+c\omega_y+d\omega_z = 0$, hence deducing $a=b=c=d=0$.
\end{proof}

Thus given an element $\alpha \in \calL$, it can be uniquely written as
\begin{equation}\label{equation-Z.combination.of.L}
\alpha = a + b \omega_x + c \omega_y + d \omega_z
\end{equation}
with the coefficients $a,b,c,d \in \calZ$ explicitly given by
\begin{equation}\label{equation-Z.combination.of.L.2}
\begin{aligned}
& a = \frac{1}{4}(\alpha + \iota_{xy}(\alpha) + \iota_{xz}(\alpha) + \iota_{yz}(\alpha)), \quad b = \frac{1}{4 \omega_x}(\alpha - \iota_{xy}(\alpha) - \iota_{xz}(\alpha) + \iota_{yz}(\alpha)), \\
& c = \frac{1}{4 \omega_y}(\alpha - \iota_{xy}(\alpha) + \iota_{xz}(\alpha) - \iota_{yz}(\alpha)), \quad d = \frac{1}{4 \omega_z}(\alpha + \iota_{xy}(\alpha) - \iota_{xz}(\alpha) - \iota_{yz}(\alpha)).
\end{aligned}
\end{equation}

In the next lemma we compute the center of $k[G]$.

\begin{lemma}\label{lemma-center.of.kG}
We have $\emph{Z}(k[G]) = k[\kappa_x,\kappa_y,\kappa_z,\xi]$.
\end{lemma}
\begin{proof}
We remark that by Proposition \ref{proposition-galois.promislow} (iv) we have $\text{Z}(k[G]) = \calK(\xi) \cap k[x^{\pm 1},y^{\pm 1},z^{\pm 1}]$, which looks like ``almost'' what we claim, but nevertheless we do not use Proposition \ref{proposition-galois.promislow} in this proof.

We note that, given $p \in k[x^{\pm 1},y^{\pm 1},z^{\pm 1}]$, the combination
$$p + p^s + p^t + p^u = p + \iota_{yz}(p) + \iota_{xz}(p) + \iota_{xy}(p) = \sum_{\gamma \in \text{Gal}(\calL/\calZ)} \gamma(p)$$
belongs to $\text{Z}(k[G])$. We thus consider the surjective linear map
$$P \colon k[x^{\pm 1},y^{\pm 1},z^{\pm 1}] \to \text{Z}(k[G]), \quad P(p) = \frac{1}{4} \sum_{\gamma \in \text{Gal}(\calL/\calZ)} \gamma(p).$$
It then suffices to show that the image of $P$ equals $k[\kappa_x,\kappa_y,\kappa_z,\xi]$. We start by checking that $P(m) \in k[\kappa_x,\kappa_y,\kappa_z,\xi]$ when $m \in k[x^{\pm 1},y^{\pm 1},z^{\pm 1}]$ is a monomial. Furthermore, since $P$ commutes with $\iota_x$, $\iota_y$, $\iota_z$, we may assume that $m \in k[x,y,z]$.

Let $\Gamma_+$ be the subgroup of $\text{Aut}(\calL/k)$ generated by $\iota_x, \iota_y,\iota_z$, and $\Gamma_-$ be the group generated by the linear maps $-\iota_x, -\iota_y,-\iota_z$. Let
$$Q(m) := \frac{1}{8} \sum_{\gamma \in \Gamma_+} \gamma(m), \quad R(m) := \frac{1}{8} \sum_{\gamma \in \Gamma_-} \gamma(m),$$
and note that $P(m) = Q(m) + R(m)$.

We first consider the case when $m$ depends on a single variable, for example $x$, thus $m = x^d$ with $d \geq 1$. We proceed by strong induction on the exponent. For $d = 1$ we have $P(m) = \kappa_x$ and the claim is true. Now assume that the result holds for all degrees up to $d \geq 1$, and we compute
$$2^{d+1} \kappa_x^{d+1} = \sum_{k=0}^{d+1} \binom{d+1}{k} x^{2k-(d+1)} = \sum_{k=0}^{d+1} \binom{d+1}{k} x^{d+1-2k}.$$
By adding these two expressions and dividing by $2$, we get
\begin{align*}
2^{d+1} \kappa_x^{d+1} & = \sum_{k=0}^{d+1} \binom{d+1}{k} P(x^{d+1-2k}) \\
& = 2 P(x^{d+1}) + \sum_{k=1}^d \binom{d+1}{k} P(x^{d+1-2k}) \\
& = 2 P(x^{d+1}) + \sum_{k=1}^{\floor{\frac{d+1}{2}}} \binom{d+1}{k} P(x^{d+1-2k}) + \sum_{k=\floor{\frac{d+1}{2}}+1}^d \binom{d+1}{k} \iota_x P(x^{2k-d-1}). \\
\end{align*}
We now use strong induction to deduce that $P(x^{d+1}) \in k[\kappa_x,\kappa_y,\kappa_z,\xi]$. The same argument works \textit{mutatis mutandis} for the variables $y$, $z$.

Let us now consider the case when $m$ depends on two variables, for example $x$ and $y$. Thus take $m = x^dy^e$ with $d,e \geq 1$. Then
$$P(m) = \frac{1}{4}(x^dy^e + x^d \bar y^e + \bar x^dy^e + \bar x^d \bar y^e) = \frac{1}{4}(x^d + \bar x^d)(y^e + \bar y^e) = P(x^d)P(y^e),$$
and thus we are done by the case of a single variable. The cases for the variables $x,z$ and $y,z$ are treated analogously.

Finally let us consider the case of three variables $m = x^dy^ez^f$, with $d,e,f \geq 1$. We note that
$$Q(m) = \frac{1}{8}(x^d + \bar x^d)(y^e + \bar y^e)(z^f + \bar z^f) = P(x^d)P(y^e)P(z^f), \quad R(m) = \frac{1}{8} (x^d - \bar x^d) (y^e - \bar y^e) (z^f - \bar z^f).$$
Since $P(m) = Q(m) + R(m)$, we only need to deal with $R(m)$. We note that the quotient
$$\alpha := \frac{x^d - \bar x^d}{x - \bar x}\frac{y^e - \bar y^e}{y - \bar y}\frac{z^f - \bar z^f}{z - \bar z}$$
is invariant under the whole group $\text{Gal}(\calL/\calK) = \Gamma_+$, and so belongs to $\calK = k(\kappa_x,\kappa_y,\kappa_z)$. In fact, this quotient can be explicitly computed to be
$$\alpha = \sum_{i=-d+1}^{d-1}x^i\sum_{j=-e+1}^{e-1}y^j\sum_{k=-f+1}^{f-1} z^k = \Big(1 + 2 \sum_{i=1}^{d-1} P(x^i)\Big)\Big(1 + 2 \sum_{j=1}^{e-1} P(y^j)\Big)\Big(1 + 2 \sum_{k=1}^{f-1} P(z^k)\Big) \in k[\kappa_x,\kappa_y,\kappa_z,\xi].$$
Therefore $R(m) = \xi \alpha \in k[\kappa_x,\kappa_y,\kappa_z,\xi]$.

This analysis proves the inclusion $\text{im}(P) \subseteq k[\kappa_x,\kappa_y,\kappa_z,\xi]$. The other inclusion is clear, since $\kappa_x,\kappa_y,\kappa_z, \xi \in \calZ \cap k[G]$.
\end{proof}

\begin{remark}\label{remark-molien.formula}
An alternative proof of Lemma \ref{lemma-center.of.kG} can be obtained by noting that $\kappa_x,\kappa_y,\kappa_z,\xi$ are invariant functions and then using Molien's formula to check that they generate the whole algebra of invariant functions.
\end{remark}

We note the following simple corollary of Lemma \ref{lemma-center.of.kG}.

\begin{lemma}\label{lemma-division}
Let $f,g \in k[\kappa_x,\kappa_y,\kappa_z]$ be such that $\frac{f}{g} \in k[x^{\pm 1},y^{\pm 1},z^{\pm 1}]$. Then $\frac{f}{g}\in k[\kappa_x,\kappa_y,\kappa_z]$. Informally, we may say that if $g|f$ in $k[x^{\pm 1},y^{\pm 1},z^{\pm 1}]$ then $g|f$ in $k[\kappa_x,\kappa_y,\kappa_z]$.
\end{lemma}
\begin{proof}
Since $\frac{f}{g} \in \calK \subseteq \calZ$ we deduce that $\frac{f}{g}\in \text{Z}(k[G])$, so by Lemma \ref{lemma-center.of.kG} we have $\frac{f}{g} = p + q\xi$ with $p,q \in k[\kappa_x,\kappa_y,\kappa_z]$. Note that $\frac{f}{g} = \iota_x(\frac{f}{g}) = p - q \xi$ which shows that $q = 0$. Thus $\frac{f}{g} = p \in k[\kappa_x,\kappa_y,\kappa_z]$ and the proof is complete.
\end{proof}

Let us gather together some information from the previous lemmas.

\begin{proposition}\label{proposition-equalities}
We have the equalities
\begin{align*}
k[x^{\pm 1},y^{\pm 1},z^{\pm 1}] & = k(x,y,z) \cap k[G]; \\
k[\kappa_x,\kappa_y,\kappa_z,\xi] & = \emph{Z}(k[G]) = k(\kappa_x,\kappa_y,\kappa_z)(\xi) \cap k[G]; \\
k[\kappa_x,\kappa_y,\kappa_z] & = k(\kappa_x,\kappa_y,\kappa_z) \cap k[x^{\pm 1},y^{\pm 1},z^{\pm 1}]. \\
\end{align*}
\end{proposition}
\begin{proof}
The first displayed equality is clear, the second one is the content of Lemma \ref{lemma-center.of.kG} and the third one follows from Lemma \ref{lemma-division}.
\end{proof}

\subsection{The \texorpdfstring{$\omega$}{}-degree of an element}\label{subsection-omega.conjugates}

Regarding Section \ref{section-general.procedure}, we will concentrate in the study of $\omega_x$-conjugates. For brevity, we will denote by $\omega$ the element $\omega_x = \frac{1}{2}(x - \ol{x})$, and by $\kappa$ the element $\kappa_x = \frac{1}{2}(x + \ol{x})$. Note that the key relation here is Formula \eqref{equation-omega.kappa}, namely
$$\kappa^2 - \omega^2 = 1.$$
We specialize Definition \ref{definition-conjugate} in this context.

\begin{definition}\label{definition-omega.conjugate}
An element $T \in \calQ$ is called an \emph{$\omega$-conjugate} if $T^2 = \omega^2$. We say that $T$ is \emph{non-trivial} if $T \neq \pm \omega$.
\end{definition}

By Theorem \ref{theorem-central.division.algebras} (2), we have that $T$ is an $\omega$-conjugate if and only if $T = \omega^U$ for some $U \in \calQ \backslash \{0\}$; thus the name \emph{$\omega$-conjugate} is appropriate.

The following proposition is an extension of Proposition \ref{proposition-general.criterion} in the particular case of the Promislow's group.

\begin{proposition}\label{proposition-omega.conjugate.to.unit}
Let $T \in \calQ$ be an $\omega$-conjugate, and define $U_T := \kappa-T$.
\begin{enumerate}[(i),leftmargin=0.7cm]
\item The element $U_T$ is a unit in $\calQ$, with inverse $U_{-T} = \kappa+T$.
\item $T \in k[G]$ if and only if $U_T \in k[G]$, in which case $U_{-T}$ also belongs in $k[G]$.
\item $T$ is a non-trivial $\omega$-conjugate if and only if $U_T$ is a non-trivial unit.
\item If the involution of $k$ restricts to the identity on its prime subfield, then $T^{\ast} = -T$ if and only if $U_T$ is a unitary.
\end{enumerate}
Conversely, let $U \in \calQ \backslash \{0\}$, and define $T_U := \omega^U = U \omega U^{-1}$.
\begin{enumerate}[(i),leftmargin=0.7cm]
\item The element $T_U$ is an $\omega$-conjugate.
\item If both $U,U^{-1} \in k[G]$, then $T_U \in k[G]$.
\item If $T_U$ is a non-trivial $\omega$-conjugate, then $U$ is a non-trivial unit. The converse is also true if we assume $U \in k[G]$.
\item If $U$ is a unitary, then $T_U^{\ast} = -T_U$.
\end{enumerate}
\end{proposition}
This is mostly Proposition \ref{proposition-general.criterion}. Nevertheless, since Proposition \ref{proposition-omega.conjugate.to.unit} gives necessary and sufficient conditions for the non-triviality of the units, we prefer to give a full detailed proof of it.
\begin{proof}[Proof of Proposition \ref{proposition-omega.conjugate.to.unit}]
Since $\kappa \in \calZ$ we have
$$U_T U_{-T} = (\kappa - T)(\kappa + T) = \kappa^2 - T^2 = \kappa^2 - \omega^2 = 1,$$
which shows that $U_T$ is a unit in $\calQ$ with inverse $U_{-T}$. Moreover, since $\kappa \in k[G]$, it is clear that $T \in k[G]$ if and only if $U_T \in k[G]$, in which case $U_{-T} \in k[G]$ as well. Furthermore, if the involution $\ast$ is the identity over the prime subfield of $k$, then a computation shows that $U_T^* = \kappa - T^{\ast}$, and so $U_T^* = U_T^{-1}$ if and only if $T^{\ast} = -T$.

Let us show now that $U_T$ is a non-trivial unit in case $T$ is a non-trivial $\omega$-conjugate. Suppose, by way of contradiction, that $U_T = \lambda g$ for some $\lambda \in k \backslash \{0\}$ and $g \in G$. After isolating $T$ in the equation $U_T = \lambda g$ and squaring, we get the equation
$$\lambda^2 g^2 - \lambda xg - \lambda \ol{x}g + 1 = 0.$$
This equation is inconsistent unless $g$ equals either $x$ or $\ol{x}$. But in such cases we get $\lambda = 1$, so $U_T = x^{\pm 1}$. This leads to $T = \pm \omega$, which are the trivial $\omega$-conjugates, a contradiction. Conversely, if $T = \pm \omega$, then
$$U_T = \kappa - T = \kappa \mp \omega = x^{\mp 1},$$
so $U_T$ is a trivial unit.\\

For the second part, we define $T_U := U \omega U^{-1}$ and compute 
$$T_U^2 = U \omega^2 U^{-1} = \omega^2$$
as $\omega^2 \in \calZ$. Hence $T_U$ is an $\omega$-conjugate. It is clear that if both $U,U^{-1} \in k[G]$ then $T_U \in k[G]$. Furthermore, if $U$ is a unitary, then $T_U = U \omega U^{\ast}$, and since $\omega^{\ast} = -\omega$ we get
$$T_U^{\ast} = U \omega^{\ast} U^{\ast} = - U \omega U^{\ast} = -T_U.$$

Let us show now that $U$ is a non-trivial unit in case $T_U$ is a non-trivial $\omega$-conjugate. Suppose, by way of contradiction, that $U = \lambda g$ for some $\lambda \in k \backslash \{0\}$ and $g \in G$. Write $g = n \tilde{f}$ for $n \in N$ and $\tilde{f} \in \tilde{F} = \{1,s,t,u\}$. Then
$$T_U = U \omega U^{-1} = n \omega^{\tilde{f}} \ol{n} \in \{\omega, -\omega\}$$
as $\omega^{\tilde{f}} \in \{\omega,-\omega\}$. Thus $T_U$ is a trivial $\omega$-conjugate, a contradiction. Conversely, suppose that $U \in k[G]$ and that $T_U = \pm \omega$. Write $U$ as
$$U = a + bs + ct + du$$
with $a,b,c,d \in k[x^{\pm 1},y^{\pm 1},z^{\pm 1}]$. The case $T_U = \omega$ implies $U^{\omega} = U$, and so $U = a + bs$. The same holds for $U^{-1}$, and so $U$ turns out to be a unit in $k[x^{\pm 1},y^{\pm 1},z^{\pm 1}][s]$. This latter ring is a subring of $k[\Z^2 \rtimes \Z]$, and such groups are known not to admit non-trivial units. Hence $U$ is a trivial unit, a contradiction.

The case $T_U = -\omega$ implies $U^{\omega} = -U$, and so $U = ct + du = (c + ds)t$. This implies that $c + ds$ is again a unit in $k[x^{\pm 1},y^{\pm 1},z^{\pm 1}][s]$, and by the same argument as above $U$ must be a trivial unit, again a contradiction. This finishes the proof.
\end{proof}

As a corollary, we get:

\begin{corollary}\label{corollary-iff.criterion}
The existence of non-trivial units in the group ring $k[G]$ is equivalent to the existence of non-trivial $\omega$-conjugates in $k[G]$.

If moreover the involution of $k$ restricts to the identity on its prime subfield, then the existence of non-trivial unitaries in $k[G]$ is equivalent to the existence of non-trivial $\omega$-conjugates in $k[G]$ which are anti-selfadjoint, i.e. an $\omega$-conjugate $T \in k[G]$ satisfying $T^{\ast} = -T$.
\end{corollary}

Corollary \ref{corollary-iff.criterion} gives a condition equivalent to the existence of units, which might conceivably be considered simpler than the determinant condition proved in \cite{determinant}, in the following sense: the determinant condition written as an equation over $\calZ$ is a single equation of degree $4$ with $16$ variables, whereas the condition for being an $\omega$-conjugate is an equation of degree $2$ with $16$ variables.

\begin{remark}\label{remark-good.observation}
Suppose that $U$ is a non-trivial unit in $k[G]$. Then by the above proposition we get that $T_U$ is a non-trivial $\omega$-conjugate in $k[G]$. Therefore $V := \kappa - T_U$ is again a non-trivial unit in $k[G]$. Moreover, its inverse is given by $V^{-1} = \kappa + T_U = 2 \kappa - V$. Thus $V$ satisfies
$$2 \kappa V - V^2 = 1.$$
From this we obtain that there are non-trivial units in $k[G]$ if and only if there are non-trivial units in $k[G]$ satisfying the equation
$$X^2 - 2 \kappa X + 1 = 0.$$
\end{remark}

We will now study $\omega$-conjugates more closely. Take $\phi : \calQ \ra \calQ$ to be any $k$-automorphism of $\calQ$ of order $2$ fixing $\omega^2$, that is $\phi(\omega^2) = \omega^2$. Then any element $\alpha \in \calQ$ can be uniquely written as
$$\alpha = \alpha_1 + \alpha_{-1}$$
where $\alpha_1,\alpha_{-1} \in \calQ$ are eigenvectors of $\phi$ with eigenvalues $1$ and $-1$, respectively. In fact, they are explicitly given by
$$\alpha_1 = \frac{1}{2}(\alpha + \phi(\alpha)), \quad \alpha_{-1} = \frac{1}{2}(\alpha - \phi(\alpha)).$$
Moreover $\phi$ preserves $\omega$-conjugacy, in the sense that it $T \in \calQ$ is an $\omega$-conjugate, then $\phi(T)$ is also an $\omega$-conjugate, simply because
$$\phi(T)^2 = \phi(T^2) = \phi(\omega^2) = \omega^2.$$

\begin{definition}\label{definition-phi.degree}
The \textit{$\phi$-degree} of an element $\alpha \in \calQ$ is the pair of natural numbers
$$\text{deg}_{\phi}(\alpha) := ( [\calZ(\alpha_1) : \calZ ],[\calZ(\alpha_{-1}) : \calZ ] ) \in \N \times \N,$$
where $\alpha_1,\alpha_{-1}$ are the elements determining the $\phi$-decomposition of $\alpha$.
\end{definition}

We note that, for any $D \in \calQ$, the division algebra $\calZ(D)$ is in fact a subfield of $\calQ$ containing $\calZ$. Thus we may use Theorem \ref{theorem-central.division.algebras} (1) to conclude that $[\calZ(D):\calZ] \leq 4$, where we recall that $[\calQ : \calZ] = 16$. Moreover, since
$$\calZ \subseteq \calZ(D) \subseteq \calQ$$
necessarily $[\calZ(D):\calZ]$ must be divisible by $[\calQ:\calZ] = 16$. We thus get a finite number of options for the $\phi$-degree of an element $\alpha \in \calQ$, namely
$$\text{deg}_{\phi}(\alpha) \in \{(1,1),(1,2),(1,4),(2,1),(2,2),(2,4),(4,1),(4,2),(4,4)\}.$$




In Proposition \ref{proposition-two.groups} we further restrict the $\phi$-degree of an $\omega$-conjugate $T \in \calQ$ in case $T_1,T_{-1} \neq 0$. 
First, we establish a technical lemma.

\begin{lemma}\label{lemma-center.omega.decomposition}
Let $T \in \calQ$ be an $\omega$-conjugate, $\phi : \calQ \ra \calQ$ a $k$-automorphism fixing $\omega^2$, and let $T = T_1 + T_{-1}$ be its $\phi$-decomposition. Then $T_1 T_{-1} = -T_{-1}T_1$, and moreover if both $T_1,T_{-1} \neq 0$ we also have the equalities
$$\emph{Z}(\calZ[T,T_{-1}]) = \calZ(T_1^2) = \calZ(T_{-1}^2)$$
and
$$[\calZ(T_1):\calZ] = [\calZ(T_{-1}):\calZ].$$
Here $\calZ[T_1,T_{-1}]$ is the subalgebra of $\calQ$ generated by $T_1$ and $T_{-1}$ over $\calZ$. As a consequence, if both $T_1,T_{-1} \neq 0$, then
$$\emph{deg}_{\phi}(T) \in \{(1,1),(2,2),(4,4)\}.$$
\end{lemma}
\begin{proof}
Let us write $\calR := \calZ[T_1,T_{-1}]$. We compute
\begin{align*}
T_1 T_{-1} + T_{-1} T_1 & = \frac{1}{4}\big[(T + \phi(T))(T - \phi(T)) + (T - \phi(T))(T + \phi(T))\big]\\
& = \frac{1}{2}(T^2 - \phi(T)^2) \\
& = \frac{1}{2}(\omega^2 - \omega^2) = 0.
\end{align*}
In particular, $T^2 = T_1^2 + T_{-1}^2 = \omega^2 \in \calZ$, hence $\calZ(T_1^2) = \calZ(T_{-1}^2)$.

Now we assume that $T_1,T_{-1} \neq 0$, so in particular $\calR$ is not commutative. A direct computation shows that $T_1^2$ commutes with $T_{-1}$, and $T_{-1}^2$ commutes with $T_1$, hence $T_1^2,T_{-1}^2 \in \text{Z}(\calR)$. This shows the inclusion
\begin{equation}\label{equation-inclusion}
\calZ(T_1^2) \subseteq \text{Z}(\calR).
\end{equation}
\begin{claim}\label{claim-1}
If either $T_1^2 \in \calZ$ or $T_{-1}^2 \in \calZ$, then $[\calR:\calZ] \leq 4$.
\end{claim}
\begin{proof}[Proof of Claim]
Due to the relation $T_1^2 + T_{-1}^2 = \omega^2 \in \calZ$, we see that both $T_1^2,T_{-1}^2 \in \calZ$. It follows that every element of $\calR$ can be written as a combination
$$z_1 + z_2 T_1 + z_3 T_{-1} + z_4 T_1 T_{-1}, \quad z_1,z_2,z_3,z_4 \in \calZ.$$
The claim follows.
\end{proof}
Now $\calZ \subseteq \calR \subseteq \calQ$, so the only possibilities for $[\calR:\calZ]$ are in fact $1,2,4,8$ and $16$. Let us check the statement of the lemma case by case.
\begin{enumerate}[1),leftmargin=0.7cm]
\item If $[\calR:\calZ] = 16$, then $\calR = \calQ$, and so $\text{Z}(\calR) = \calZ \subseteq \calZ(T_1^2)$. We are done by \eqref{equation-inclusion}.
\item If $[\calR:\calZ] = 8$, then by Theorem \ref{theorem-central.division.algebras} (1) the quantity $[\calR : \text{Z}(\calR)]$ must be a square dividing $8$, thus it is either $1$ or $4$. But $\calR$ is not commutative, so necessarily $[\calR:\text{Z}(\calR)] = 4$ and so $[\text{Z}(\calR):\calZ] = 2$. Thus by \eqref{equation-inclusion} it is enough to argue that $T_1^2 \notin \calZ$. But this follows from Claim \ref{claim-1}.
\item If $[\calR:\calZ] = 4$, then again Theorem \ref{theorem-central.division.algebras} (1) implies that $\text{Z}(\calR) = \calZ \subseteq \calZ(T_1^2)$. We are again done by \eqref{equation-inclusion}.
\item The cases $[\calR:\calZ] = 1$ and $2$ lead to the conclusion that $\calR$ is commutative, which is a contradiction.
\end{enumerate}
This analysis finishes the  first equality. For the second one, simply note that 
$$[\calZ(T_1) : \calZ] = [\calZ(T_1):\calZ(T_1^2)][\calZ(T_1^2):\calZ] = [\calZ(T_1):\calZ(T_1^2)][\text{Z}(\calR):\calZ],$$
$$[\calZ(T_{-1}) : \calZ] = [\calZ(T_{-1}):\calZ(T_{-1}^2)][\calZ(T_{-1}^2):\calZ] = [\calZ(T_{-1}):\calZ(T_{-1}^2)][\text{Z}(\calR):\calZ],$$
and moreover $[\calZ(T_1):\calZ(T_1^2)] = [\calZ(T_{-1}):\calZ(T_{-1}^2)] = 2$, since both $\calZ(T_1^2) = \calZ(T_{-1}^2) = \text{Z}(\calR)$ and $T_1$ does not commute with $T_{-1}$. Therefore we conclude that
$$[\calZ(T_1):\calZ] = [\calZ(T_{-1}):\calZ].$$
The statement of the $\phi$-degrees now follows.
\end{proof}

The first part of the following proposition shows that non-trivial $\omega$-conjugates fall naturally into two groups, assuming that both $T_1,T_{-1} \neq 0$.

\begin{proposition}\label{proposition-two.groups}
Let $T \in \calQ$ be an $\omega$-conjugate, $\phi : \calQ \ra \calQ$ a $k$-automorphism fixing $\omega^2$, and let $T = T_1 + T_{-1}$ be its $\phi$-decomposition. Assume that $T_1,T_{-1} \neq 0$.

Then either $T$ is a trivial $\omega$-conjugate or $\emph{deg}_{\phi}(T) \in \{(2,2),(4,4)\}$. Moreover, in the particular case that $T$ is non-trivial, the following are equivalent.
\begin{enumerate}[(i),leftmargin=0.7cm]
\item $\deg_{\phi}(T) = (2,2)$.
\item $T_1^2, T_{-1}^2 \in \calZ$.
\item $[\calZ[T_1,T_{-1}] : \calZ] = 4$.
\end{enumerate}
\end{proposition}
\begin{proof}
By Lemma \ref{lemma-center.omega.decomposition}, $\text{deg}_{\phi}(T) \in \{(1,1),(2,2),(4,4)\}$. Therefore to establish the first part of the proposition it is enough to argue that $\deg_{\phi}(T) \neq (1,1)$ in case $T$ is a non-trivial $\omega$-conjugate. However, if $\deg_{\phi}(T) = (1,1)$, then $T \in \calZ$ and the equation $X^2 = \omega^2$ would have more than two solutions in the field $\calZ(\omega)$, namely $\{\pm \omega,\pm T\}$, which cannot happen. This concludes the first part of the proposition.

For the second part of the proposition, suppose that $T$ is a non-trivial $\omega$-conjugate. We first prove (i) $\Rightarrow$ (ii). We note that $\deg_{\phi}(T) = (2,2)$ implies that we can find $\alpha,\beta \in \calZ$ such that $T_1^2 = \alpha T_1 +\beta$. Since $T_1^2$ commutes with $T_{-1}$, we deduce that
$$\alpha T_1 + \beta = (\alpha T_1 + \beta)^{T_{-1}} = - \alpha T_1 + \beta,$$
i.e. $T_1^2 = \beta \in \calZ$. Therefore $T_{-1}^2 = \omega^2 - T_1^2 \in \calZ$ also, and this proves (ii).

To establish (ii) $\Rightarrow$ (iii), let us assume that both $T_1^2,T_{-1}^2 \in \calZ$. By Claim \ref{claim-1} we get $[\calZ[T_1,T_{-1}] : \calZ] \in\{1,2,4\}$. Since $T_1,T_{-1} \neq 0$ by assumption, $\calZ[T_1,T_{-1}]$ is not commutative. Hence necessarily $[\calZ[T_1,T_{-1}] : \calZ] = 4$.

Finally we establish (iii) $\Rightarrow$ (i), so assume $[\calZ[T_1,T_{-1}] : \calZ] = 4$. Then by Lemma \ref{lemma-center.omega.decomposition} we have
$$4 = [\calZ[T_1,T_{-1}]:\calZ] = [\calZ[T_1,T_{-1}] : \text{Z}(\calZ[T_1,T_{-1}])] [\calZ(T_1^2):\calZ].$$
Since $\calZ[T_1,T_{-1}]$ is not commutative, using Theorem \ref{theorem-central.division.algebras} (1) we deduce that $[\calZ(T_1^2):\calZ] = 1$, i.e. $T_1^2 \in \calZ$. Since $T_1 \notin \calZ$, necessarily $[\calZ(T_1):\calZ] = 2$ and thus, by the first part of the proposition, we must have $\deg_{\phi}(T) = (2,2)$. This finishes the proof.
\end{proof}

Let us note an important improvement of Proposition \ref{proposition-two.groups} in a particular case. Recall that conjugation by $\omega$ is an automorphism of $\calQ$ of order $2$ which preserves $k[G]$. Explicitly, we have
$$(a + bs + ct + du)^{\omega} = a + bs - ct - du$$
for $a,b,c,d \in \calL = k(x,y,z)$. We will denote by $\omega$ itself the automorphism of $\calQ$ given by conjugation by $\omega$. The $\omega$-decomposition of $T = a + bs + ct + du$ is given explicitly by $T_1 = a + bs$ and $T_{-1} = ct + du$.

\begin{proposition}\label{proposition-omega.two.groups}
Let $T \in \calQ$ be an $\omega$-conjugate and $T = T_1 + T_{-1}$ be its $\omega$-decomposition. If $T$ is non-trivial, then necessarily $T_1,T_{-1} \neq 0$.
\end{proposition}
Combining Propositions \ref{proposition-omega.two.groups} and \ref{proposition-two.groups}, we conclude that an $\omega$-conjugate $T \in \calQ$ is either trivial or $\text{deg}_{\omega}(T) \in \{(2,2),(4,4)\}$, in which case the following equivalences apply:
\begin{enumerate}[(1),leftmargin=0.7cm]
\item $\deg_{\omega}(T) = (2,2)$.
\item $T_1^2, T_{-1}^2 \in \calZ$.
\item $[\calZ[T_1,T_{-1}] : \calZ] = 4$.
\end{enumerate}
\begin{proof}[Proof of Proposition \ref{proposition-omega.two.groups}]
If we denote by $\ol{k}$ the algebraic closure of $k$, then the inclusion $k \subseteq \ol{k}$ gives an inclusion $k[G] \hookrightarrow \ol{k}[G]$, which in turn extends to an inclusion of the corresponding Ore field of fractions. Therefore we may assume, without loss of generality, that $k$ is algebraically closed.

We first show that we can never have $T_1 = 0$. By way of contradiction, let us assume that $T_1 = 0$. We have $\omega T_{-1} = -T_{-1} \omega$ and $T_{-1}^2 = \omega^2$, thus $k(\omega^2)[\omega,T_{-1}]$ is a division algebra with $[k(\omega^2)[\omega,T_{-1}]: k(\omega^2) ] = 4$, and whose center is $k(\omega^2)$. However, Tsen's Theorem (see \cite{Tsen}) shows that the field of rational functions in one variable over an algebraically closed field cannot be the center of a finite-dimensional division algebra. This contradiction shows that $T_1 \neq 0$.

Now, in case $T_{-1} = 0$ and $T = T_1 \neq \pm \omega$, then $\calZ[\omega,T_1]$ is a field in which the equation $X^2 =\omega^2$ has at least $4$ different solutions, namely $\{\pm \omega, \pm T_1\}$, which is a contradiction. This concludes the proof of the proposition.
\end{proof}

\begin{definition}\label{definition-types.units}
A non-trivial unit $U \in k[G]$ is said to be of \textit{type 2} if $\text{deg}_{\omega}(\omega^U) = (2,2)$. Otherwise it is said to be of \text{type 4}, i.e. in case $\text{deg}_{\omega}(\omega^U) = (4,4)$.
\end{definition}

\begin{remark}\label{remark-on.omega.conjugates}
\text{ }
\begin{enumerate}[1),leftmargin=0.7cm]
\item Let us see that there are $\omega$-conjugates in $\calQ$ whose $\omega$-degree is $(4,4)$. Indeed, let $k = \Q$ and let $U:= (1+x)+u$. Then a direct check shows that
$$U^{-1} = \frac{1 + \bar x}{2 + x + \bar x - z} - \frac{1}{2 + x + \bar x - z}u,$$
and so
$$T := \omega^U = \frac{2 + x + \bar x + z}{2 + x + \bar x - z} \omega - \frac{2(1+x)\omega}{2 + x + \bar x - z}u.$$
Note that here, if $T = T_1 + T_{-1}$ denotes the $\omega$-decomposition of $T$, we obtain
$$T_1 = \frac{2 + x + \bar x + z}{2 + x + \bar x - z} \omega$$
and another computation shows that $T_1^2$ is not invariant under $\iota_z$, so that $T_1^2 \notin \calZ$.
\item On the other hand, there are no non-trivial $\omega$-conjugates in $k[G]$ of the form $T = a+bt$, where $a,b \in k[x^{\pm 1},y^{\pm 1}, z^{\pm 1}]$, by the familiar argument using the unique product property. Indeed, we can replace every occurrence of $y$ with $t^2$ to express $T$ as a polynomial in $t$ with coefficients in $k[x^{\pm 1},z^{\pm}]$, and it is easy to argue that $T^2 \in k[x^{\pm 1}]$ implies that in fact $T \in k[x^{\pm 1}]$, which then implies that $T = \pm \omega$.
\item In Proposition \ref{proposition-omega.22.examples} below we use the results from \cite{Murray} to find non-trivial $\omega$-conjugates of $\omega$-degree $(2,2)$ in $k[G]$, whenever $\text{char}(k) > 2$. However, we were not able to find $\omega$-conjugates in $k[G]$ with $\omega$-degree $(4,4)$.
\end{enumerate}
\end{remark}

For the purpose of searching for $\omega$-conjugates it is convenient to assume some extra hypotheses about the coefficients. The following proposition shows that, if there are non-trivial $\omega$-conjugates in $k[G]$ with $\omega$-degree $(2,2)$, then there are also non-trivial $\omega$-conjugates with the same $\omega$-degree, but whose $\omega$-decomposition $T_1 + T_{-1}$ satisfies $T_1 \omega \in \calZ$. 

\begin{proposition}\label{proposition-special.22}
Suppose that $T \in \calQ$ is an $\omega$-conjugate with $\deg_{\omega}(T) = (2,2)$, and consider the unit $U_T = \kappa - T$ given in Proposition \ref{proposition-omega.conjugate.to.unit}. Then the element
$$S := \omega^{U_T}$$
has the following properties:
\begin{enumerate}[(i),leftmargin=0.7cm]
\item $S$ is an $\omega$-conjugate with $\deg_{\omega}(S) = (2,2)$.
\item If $S = S_1 + S_{-1}$ is the $\omega$-decomposition of $S$, then $S_1 = \alpha \omega$ for some $\alpha \in \calZ$. If moreover $T \in k[G]$ then we can take $\alpha \in k[G]$.
\item If the involution of $k$ restricts to the identity on its prime subfield and $T^{\ast} = -T$, then also $S^{\ast} = -S$.
\end{enumerate}
\end{proposition}
\begin{proof}
Let $T = T_1 + T_{-1}$ be the $\omega$-decomposition of $T$. We have
\begin{align*}
S_1 & = \frac{1}{2}(S+S^{\omega}) \\
& = \frac{1}{2}(\omega^{U_T} + \omega^{\omega U_T}) \\
& = \frac{1}{2} \Big( (\kappa - T_1 - T_{-1}) (\kappa + T_1 - T_{-1}) + (\kappa - T_1 + T_{-1})(\kappa + T_1 + T_{-1}) \Big) \omega \\
& = (\kappa^2 - T_1^2 + T_{-1}^2) \omega \\
& = (\kappa^2 - \omega^2 + 2 T_{-1}^2) \omega = (1 + 2 T_{-1}^2) \omega
\end{align*}
and
\begin{align*}
S_{-1} & = \frac{1}{2}(S-S^{\omega}) \\
& = \frac{1}{2}(\omega^{U_T} - \omega^{\omega U_T}) \\
& = \frac{1}{2} \Big( (\kappa - T_1 - T_{-1}) (\kappa + T_1 - T_{-1}) - (\kappa - T_1 + T_{-1})(\kappa + T_1 + T_{-1}) \Big) \omega \\
& = 2(T_1T_{-1} - \kappa T_{-1}) \omega = 2(T_1 - \kappa)T_{-1} \omega.
\end{align*}
Since $\deg_{\omega}(T) = (2,2)$, we see that $T_1 - \kappa \neq 0$, hence $S_{-1} \neq 0$ and therefore $\text{deg}_{\omega}(S)$ is either $(2,2)$ or $(4,4)$. By Propositions \ref{proposition-two.groups} and \ref{proposition-omega.two.groups}, $1+2T_{-1}^2 \in \calZ$, and hence $[\calZ(S_1):\calZ] \leq 2$. This establishes both (i) and (ii) with $\alpha = 1 + 2T_{-1}^2$. Observe that if $T \in k[G]$ then also $T^{\omega} \in k[G]$, so $T_{-1} \in k[G]$ and $\alpha \in k[G]$.

As for (iii), we recall (see the proof of Corollary \ref{corollary-iff.criterion}) that if $T$ is anti-selfadjoint then $U_T$ is a unitary. In this case, we compute
$$S^{\ast} = (\omega^{U_T})^{\ast} = (U_T \omega U_T^{\ast})^{\ast} = U_T \omega^{\ast} U_T^{\ast}.$$
Since $\omega^{\ast} = -\omega$, this concludes the proof that $S^{\ast} = -S$, and the proof of the proposition.
\end{proof}

As corollaries, we get the following results.

\begin{corollary}\label{corollary-special.22}
The following hold true:
\begin{enumerate}[(a),leftmargin=0.7cm]
\item There exists a non-trivial unit $U \in k[G]$ of type $2$ if and only if there exists a non-trivial $\omega$-conjugate $S \in k[G]$ with $\emph{deg}_{\omega}(S) = (2,2)$ of the form
$$S = \alpha \omega + ct + du,$$
with $\alpha \in k[\kappa_x,\kappa_y,\kappa_z,\xi]$ and $c,d \in k[x^{\pm 1},y^{\pm 1},z^{\pm 1}]$.
\item Assuming that the involution of $k$ restricts to the identity on its prime subfield, there exists a non-trivial unitary $U \in k[G]$ of type $2$ if and only if there exists a non-trivial anti-selfadjoint $\omega$-conjugate $S \in k[G]$ with $\emph{deg}_{\omega}(S) = (2,2)$ of the form
$$S = \alpha \omega + ct + du,$$
with $\alpha \in k[\kappa_x,\kappa_y,\kappa_z,\xi]$ and $c,d \in k[x^{\pm 1},y^{\pm 1},z^{\pm 1}]$.
\end{enumerate}
\end{corollary}
\begin{proof}
Given such a unit $U$ of type $2$, we let $T := \omega^U$, which has $\omega$-degree $(2,2)$ by assumption. In particular $T \in k[G]$ is non-trivial, and so the unit $U_T := \kappa - T$ is a non-trivial unit in $k[G]$ by Proposition \ref{proposition-omega.conjugate.to.unit}. Thus $S := \omega^{U_T}$ is also a non-trivial $\omega$-conjugate in $k[G]$ by the same proposition, with $\omega$-degree $(2,2)$ and of the required form by Propositions \ref{proposition-special.22} and \ref{proposition-equalities}.

Conversely, let $S \in k[G]$ be a non-trivial $\omega$-conjugate of the stated form, with $\omega$-degree $(2,2)$. Then the associated unit $U_S := \kappa - S$ is a non-trivial unit in $k[G]$ of type $2$ by Propositions \ref{proposition-omega.conjugate.to.unit} and \ref{proposition-special.22}. This proves part (a).

Part (b) is completely analogous, taking into account the corresponding statements relating unitarity of $U_T$ and anti-selfadjointness of $T$.
\end{proof}

\begin{corollary}\label{corollary-special.22.2}
There exists a non-trivial unit $U \in k[G]$ of type $2$ if and only if there exists a non-trivial unit $V \in k[G]$ of the form
$$V = (\kappa - \alpha \omega) - ct - du$$
where $\alpha \in k[\kappa_x,\kappa_y,\kappa_z,\xi]$ and $c,d \in k[x^{\pm 1},y^{\pm 1},z^{\pm 1}]$, whose inverse is $V^{-1} = (\kappa + \alpha \omega) + ct + du$.
\end{corollary}
\begin{proof}
If $U \in k[G]$ is a non-trivial unit of type $2$, then by Corollary \ref{corollary-special.22} there is a non-trivial $\omega$-conjugate $S \in k[G]$ of the form $S = \alpha \omega + ct + du$ with $\alpha \in k[\kappa_x,\kappa_y,\kappa_z,\xi]$ and $c,d \in k[x^{\pm 1},y^{\pm 1},z^{\pm 1}]$. By Proposition \ref{proposition-omega.conjugate.to.unit}, the unit $V := \kappa - S$ is non-trivial in $k[G]$ with inverse $V^{-1} = \kappa + S$, which is of the desired form.

Conversely, suppose that $V = (\kappa - \alpha \omega) - ct - du$ is a non-trivial unit in $k[G]$ with inverse $V^{-1} = (\kappa + \alpha \omega) + ct + du$. We compute
\begin{align*}
\omega^V & = V \omega V^{-1} \\
& = ((\kappa^2 - \alpha^2 \omega^2) + cc^t y + dd^u z)\omega + (cd^t \ol{x} \ol{z} + c^u d \ol{y}) \omega s + 2 (\kappa - \alpha \omega)\omega ct + 2(\kappa - \alpha \omega)\omega dt \\
& = ( 1 + 2 \omega^2(1 - \alpha^2))\omega + 2 (\kappa - \alpha \omega)\omega ct + 2(\kappa - \alpha \omega)\omega dt,
\end{align*}
where we have used that $VV^{-1} = 1$. From here we conclude that $(\omega^V)_1^2 \in \calZ$, and so $\text{deg}_{\omega}(\omega^V) = (2,2)$ by Proposition \ref{proposition-omega.two.groups}.
\end{proof}


The following proposition, which uses the existence of non-trivial units in $k[G]$ for the case $\text{char}(k) \neq 0$ recently proved in \cite{Gardam,Murray}, proves existence of non-trivial $\omega$-conjugates of the first class, namely having $\omega$-degree $(2,2)$.

\begin{proposition}\label{proposition-omega.22.examples}
Let $\Gamma = G$ be the Promislow's group. If $\emph{char}(k) > 2$ then there exists a non-trivial unit in $k[G]$ of type $2$.
\end{proposition}
\begin{proof}
Write $p := \text{char}(k) > 2$. In \cite[Theorem 3]{Murray} the author constructs non-trivial units in $k[G]$. Explicitly, these units are given as follows.

For any choice of parameters $m,n \in \Z$, define $h := (1-z^{1-2n})^{p-2} \in k[x^{\pm 1},y^{\pm 1},z^{\pm 1}]$ and elements $a,b,c,d \in k[x^{\pm 1},y^{\pm 1},z^{\pm 1}]$ by
\begin{equation*}
\begin{aligned}
& a := (1+x)(1+y)(z^n + z^{1-n}) h, \\
& b := z^m[(1+x)(\ol{x}+\ol{y}) + (1+\ol{y})(1+z^{2n-1})]h, \\
& c := z^m[(1+\ol{y})(x+y)z^n + (1+x)(z^n+z^{1-n})]h, \\
& d := z^{2n-1} + (4+x+\ol{x}+y+\ol{y})h.
\end{aligned}
\end{equation*}
Then $U := a + bs + ct + du$ is a non-trivial unit in $k[G]$, with inverse given by
$$U^{-1} = a' + b's + c't + d'u,$$
where
$$a' = \ol{x} a^s, \quad b' = -\ol{x} b, \quad c' = -\ol{y}c, \quad d' = \ol{z} d^s.$$
By Proposition \ref{proposition-omega.conjugate.to.unit}, the element $T := \omega^U$ is a non-trivial $\omega$-conjugate in $k[G]$. We will show that $\text{deg}_{\omega}(T) = (2,2)$, which will prove the proposition. We first obtain the $\omega$-decomposition of $T$:
$$T = \omega^U = U \omega U^{-1} = (a+bs+ct+du)(a'+b's-c't-d'u)\omega,$$
and so
\begin{equation*}
\begin{aligned}
T_1 = \frac{1}{2}(T+T^{\omega}) = & \Big[ (aa'+xb(b')^s-yc(c')^t-zd(d')^u) \\
& + (ab' + b(a')^s-\ol{x} \text{ } \ol{z} c(d')^t-\ol{y}d(c')^u) s \Big] \omega, \\
T_{-1} = \frac{1}{2}(T-T^{\omega}) = & \Big[ (-ac'-xb(d')^s+c(a')^t+\ol{y}zd(b')^u) t \\
& + (-ad' -b(c')^s+\ol{x}y\ol{z}c(b')^t+d(a')^u) u \Big] \omega.
\end{aligned}
\end{equation*}
Since $(a')^s = \ol{x} a$, $(c')^u = -\bar{x}yc$ and $(d')^t = zd$, we see that
$$ab' + b(a')^s = -\ol{x}ab + \ol{x}ba = 0$$
and also
$$\ol{x} \text{ }\ol{z} c(d')^t+\ol{y}d(c')^u = \ol{x} cd - \ol{x}dc = 0.$$
Hence the expression accompanying the element $s$ in the expression of $T_1$ is exactly $0$, which leads to
$$T_1 = (aa'+xb(b')^s-yc(c')^t-zd(d')^u) \omega =: \alpha \omega.$$
We compute the expression \eqref{equation-Z.combination.of.L} of $\alpha \in \calL$ as a $\calZ$-linear combination using the formulas \eqref{equation-Z.combination.of.L.2}. In fact, a direct computation shows that
\begin{align*}
& \alpha - \iota_{xy}(\alpha) - \iota_{xz}(\alpha) + \iota_{yz}(\alpha) = 0, \\
& \alpha - \iota_{xy}(\alpha) + \iota_{xz}(\alpha) - \iota_{yz}(\alpha) = 0, \\
& \alpha + \iota_{xy}(\alpha) - \iota_{xz}(\alpha) - \iota_{yz}(\alpha) = 0.
\end{align*}
We thus conclude that $\alpha$ is a central element in $\calQ$. Therefore $T_1^2 = \alpha^2 \omega^2 \in \calZ$, and this already implies that $\text{deg}_{\omega}(T) = (2,2)$ by Proposition \ref{proposition-omega.two.groups}. This finishes the proof.
\end{proof}

It is an interesting question whether we can find non-trivial units in $k[G]$ of type $2$ in the case $\text{char}(k) = 0$. Also, when the involution of $k$ restricts to the identity on its prime subfield, it is also interesting to ask about non-trivial unitaries in $k[G]$ of type $2$. These questions are handled in the next section.

\subsection{Criteria for existence of \texorpdfstring{$\omega$}{}-conjugates with \texorpdfstring{$\omega$}{}-degree \texorpdfstring{$(2,2)$}{}}\label{subsection-criteria.existence.omega.conjugates}

In this section we provide criteria for existence of non-trivial units in $k[G]$ of type $2$. They will be stated for an arbitrary field $k$ with $\text{char}(k) \neq 2$, although in view of Proposition \ref{proposition-omega.22.examples} they are interesting only when $\text{char}(k) = 0$.

In view of Corollary \ref{corollary-special.22}, we will study non-trivial $\omega$-conjugates $T \in \calQ$ of the form
\begin{equation}\label{equation-S.form}
T = a \omega + ct +du,
\end{equation}
where $a \in k[\kappa_x,\kappa_y,\kappa_z,\xi] \subseteq \calZ$ and $c,t \in k[x^{\pm 1},y^{\pm 1},z^{\pm 1}] \subseteq \calL$. Non-triviality of $T$ implies that both $c,d \neq 0$. Indeed, if $c = 0$ then $T = a \omega + du$ and $U_T := \kappa - T = (\kappa - a \omega) - du$ would be a non-trivial unit in $k[x^{\pm 1},y^{\pm 1},z^{\pm 1}][u]$. A similar argument as the one given in the proof of Proposition \ref{proposition-omega.conjugate.to.unit} implies that $U_T$ could not be non-trivial, giving a contradiction. Similarly, we can prove that $d \neq 0$. The criteria are obtained by considering the equation
$$T^2 = \omega^2$$
and writing it using explicit bases of $\calQ$ over $\calZ$ which are convenient for computations. Let us start by introducing such bases.

The first column of the following table introduces a basis of $\calL$ over $\calZ$ convenient for studying the coefficient $c$.

\begin{table}[h]
\renewcommand*{\arraystretch}{1.4}
\centering
\begin{tabular}{c|c|c|c|c|c|}
\cline{2-6}
 & ${}^s$ & ${ }^t$ & ${ }^u$ & $\gamma_i^2 \bar x y$ & $(\gamma_it)^*$ \\
\hline
\hline
\multicolumn{1}{|c||}{$\gamma_1 := (1+x)(1-\bar y)$} & $-y$ & $\phantom{-}\bar x$ & $-\bar x y$ & $4(\kappa_x+1)(\kappa_y-1)$ & $-\gamma_1 t$ \\
\hline
\multicolumn{1}{|c||}{$\gamma_2 := (1-x)(1+\bar y)$} & $\phantom{-}y$ & $-\bar x$ & $-\bar x y$ & $4(\kappa_x-1)(\kappa_y+1)$ & $\phantom{-}\gamma_2 t$ \\
\hline
\multicolumn{1}{|c||}{$\gamma_3: = (1+x)(1+\bar y)$} & $\phantom{-}y$ & $\phantom{-}\bar x$ & $\phantom{-}\bar x y$ & $4(\kappa_x+1)(\kappa_y+1)$ & $\phantom{-}\gamma_3 t$ \\
\hline
\multicolumn{1}{|c||}{$\gamma_4 := (1-x)(1-\bar y)$} & $-y$ & $-\bar x$ & $\phantom{-}\bar x y$ & $4(\kappa_x-1)(\kappa_y-1)$ & $-\gamma_4 t$ \\
\hline
\end{tabular}
\caption{First basis: $\gamma$-elements.}
\label{table-first.basis}
\end{table}


The columns 2-4 of Table \ref{table-first.basis} show how conjugation acts on the elements of the basis, e.g. $\gamma_1^s = -y \cdot \gamma_1$. Note that those conjugation relations actually prove that $\gamma_i$ form a basis. The final two columns are for a future reference. We briefly note that $y, \bar x$ and $\bar x y$ are \emph{not} central elements.

The second basis, for studying the coefficient $d$, is presented in the following table.

\begin{table}[h]
\renewcommand*{\arraystretch}{1.4}
\centering
\begin{tabular}{c|c|c|c|c|c|}
\cline{2-6}
 & ${}^s$ & ${ }^t$ & ${ }^u$ & $\delta_i^2 z$ & $(\delta_i u)^*$ \\
\hline
\hline
\multicolumn{1}{|c||}{$\delta_1 := (1-\bar z)\phantom{(x-\bar x)}$} & $-z$ & $-z$ & $\phantom{-}1$ & $2(\kappa_z-1)$ & $-\delta_1 u$ \\
\hline
\multicolumn{1}{|c||}{$\delta_2 := (x-\bar x)(1+\bar z)$} & $\phantom{-}z$ & $-z$ & $-1$ & $8\omega_x^2(\kappa_z+1)$ & $\phantom{-}\delta_2 u$ \\
\hline
\multicolumn{1}{|c||}{$\delta_3: = (1+\bar z)\phantom{(x-\bar x)}$} & $\phantom{-}z$ & $\phantom{-}z$ & $\phantom{-}1$ & $2(\kappa_z+1)$ & $\phantom{-}\delta_3 u$ \\
\hline
\multicolumn{1}{|c||}{$\delta_4 := (x-\bar x)(1-\bar z)$} & $-z$ & $\phantom{-}z$ & $-1$ & $8 \omega_x^2(\kappa_z-1)$ & $-\delta_4 u$ \\
\hline
\end{tabular}
\caption{Second basis: $\delta$-elements.}
\label{table-second.basis}
\end{table}


Let us also note the following relations for future reference.

\begin{equation}\label{equation-basis.relations}
\begin{aligned}
\frac{\gamma_2\delta_2}{\gamma_3\delta_3} = \frac{\gamma_4\delta_4}{\gamma_1\delta_1} & = \frac{2\omega_x(1-x)}{(1+x)} = \frac{2\omega_x(1-x^2)}{1+2x+x^2} = \frac{-4\omega_x^2}{2\kappa_x+2} = -2(\kappa_x-1), \\
\frac{\gamma_1\delta_2}{\gamma_4\delta_3} = \frac{\gamma_3\delta_4}{\gamma_2\delta_1} & = \frac{2\omega_x(1+x)}{1-x} = -2(\kappa_x+1).
\end{aligned}
\end{equation}

From now on, we write
$$c = c_1+c_2+c_3+c_4, \quad d = d_1+d_2+d_3+d_4,$$
where, for all $1 \leq i \leq 4$, we have $c_i = C_i \gamma_i$, $d_i =D_i\delta_i$ for some $C_i,D_i\in \calZ$. We start with a technical lemma.

\begin{lemma}\label{lemma-equivalence.Laurent}
We have $c \in k[x^{\pm 1},y^{\pm 1},z^{\pm 1}]$ if and only if each $c_i \in k[x^{\pm 1},y^{\pm 1},z^{\pm 1}]$. Similarly, $d \in k[x^{\pm 1},y^{\pm 1},z^{\pm 1}]$ if and only if each $d_i \in k[x^{\pm 1},y^{\pm 1},z^{\pm 1}]$.
\end{lemma}
\begin{proof}
We will deal with the coefficient $c$, being the other one $d$ completely analogous. After conjugation by $s,t,u$, we get
$$\begin{cases}
c = c_1 + c_2 + c_3 + c_4,\\
c^s = y(-c_1 + c_2 + c_3 - c_4), \\
c^t = \bar x(c_1 - c_2 + c_3 - c_4), \\
c^u = \bar x y(-c_1 - c_2 + c_3 + c_4).
\end{cases}$$
Therefore we get the formulas
\begin{equation}
\begin{aligned}
& c_1 = \frac{c - \bar y c^s + x c^t - x \bar y c^u}{4}, \quad c_2 = \frac{c + \bar y c^s - x c^t - x \bar y c^u}{4}, \\
& c_3 = \frac{c + \bar y c^s + x c^t + x \bar y c^u}{4}, \quad c_4 = \frac{c - \bar y c^s - x c^t + x \bar y c^u}{4}.
\end{aligned}
\end{equation}
From here the lemma easily follows, since $c \in k[x^{\pm 1},y^{\pm 1},z^{\pm 1}]$ if and only if $c^s,c^t,c^u \in k[x^{\pm 1},y^{\pm 1},z^{\pm 1}]$.
\end{proof}

The next proposition gives necessary and sufficient conditions for $T$ as in \eqref{equation-S.form} so that $T^2 \in \calZ$.

\begin{proposition}\label{proposition-equivalence.central}
Let $T$ be as in \eqref{equation-S.form}. Then the following conditions are equivalent.
\begin{enumerate}[(i),leftmargin=0.7cm]
\item $T^2 \in \calZ$.
\item We have
\begin{equation}\label{equation-condition.central}
c d^t \bar x \bar z + c^u d \bar y = 0.
\end{equation}
\item The equation
$$(\kappa-1)C_2C_4 - (\kappa+1)C_1C_3 = 0$$
is satisfied and (at least) one of the following is true:
\begin{enumerate}[($I$),leftmargin=0.7cm]
\item In case either $C_1 \neq 0$ or $C_4 \neq 0$, there exist $\alpha,\beta \in \calZ$ such that
\begin{align*}
D_1 = - 2\beta(\kappa-1) C_4, && D_4 = \beta C_1, && D_2 = \alpha C_4, && D_3 = - 2\alpha (\kappa+1)C_1;
\end{align*}
\item In case either $C_2 \neq 0$ or $C_3 \neq 0$, there exist $\alpha, \beta \in \calZ$ such that
\begin{align*}
D_1 = - 2\beta(\kappa+1) C_3, && D_4 = \beta C_2, && D_2 = \alpha C_3, && D_3 = - 2\alpha (\kappa-1)C_2.
\end{align*}
\end{enumerate}
\end{enumerate}
If any of these conditions hold, then
$$T^2 = a^2 \omega^2 + (c_1^2 -c_2^2 +c_3^2 -c_4^2)\bar x y  + (d_1^2 -d_2^2 +d_3^2 - d_4^2)z$$
and furthermore all nine summands are in $\calZ$.
\end{proposition}
\begin{proof}
In this proof we will freely use the conjugation tables \eqref{table-first.basis} and \eqref{table-second.basis}.

For the implication $(i) \Rightarrow (ii)$, we simply compute
$$T^2 = (a \omega + ct + du)^2 = (a^2 \omega^2 + cc^ty + dd^uz) + (cd^t \bar x \bar z +c^ud \bar y)s.$$
If we require $T^2 \in \calZ$, we necessarily need
$$cd^t \bar x \bar z +c^ud \bar y = 0,$$
which is the statement in $(ii)$.

Now we show the equivalence $(ii) \Leftrightarrow (iii)$. We have
\begin{align*}
& c^u = \bar x y (-c_1-c_2+c_3+c_4), \\
& d^t = z(-d_1-d_2+d_3+d_4),
\end{align*}
thus equation \eqref{equation-condition.central} becomes
\begin{align*}
0 & = c d^t \bar x y \bar z + c^u d \\
& = (c_1+c_2+c_3+c_4)(-d_1-d_2+d_3+d_4) \bar x y + (-c_1-c_2+c_3+c_4)(d_1+d_2+d_3+d_4) \bar x y \\
& = 2 \bar x y [(c_3+c_4)(d_3+d_4) - (c_1+c_2)(d_1+d_2)].
\end{align*}
This is equivalent to the equation
$$(c_1 + c_2) (d_1 + d_2)  =  (c_3 + c_4) (d_3 + d_4).$$
By conjugating this last equation with $s$, $t$ and $u$, we get the following system of non-linear equations:
$$\begin{cases}
(c_1 + c_2) (d_1 + d_2)  =  (c_3 + c_4) (d_3 + d_4), \\
(c_1 - c_2) (d_1 - d_2)  =  (c_3 - c_4) (d_3 - d_4), \\
(-c_1 + c_2) (d_1 + d_2)  =  (c_3 - c_4) (d_3 + d_4), \\
(c_1 + c_2) (-d_1 + d_2)  =  (c_3 + c_4) (d_3 - d_4).
\end{cases}$$
This system is equivalent to the following system of equations:
$$\begin{cases}
c_1d_1 - c_4d_4 = 0, \\
c_1d_2 - c_4d_3 = 0, \\
c_2d_1 - c_3d_4 = 0, \\
c_2d_2 - c_3d_3 = 0.
\end{cases}$$
Using the definition of the $c_i$'s and the $d_i$'s, and the relations \eqref{equation-basis.relations}, those equations are equivalent to the following equations:
\begin{align*}
2(\kappa-1) C_2D_2 + C_3D_3 &= 0
&
C_1D_1+2(\kappa-1) C_4D_4 &=0
\\
2(\kappa+1)C_1D_2 + C_4D_3 &= 0
&
C_2D_1 + 2(\kappa+1)C_3D_4 &=0
\end{align*}
Note that, given $C_1,C_2,C_3,C_4$, these are in fact two systems of equations, the first system to determine $D_2,D_3$ and the other system to determine $D_1,D_4$. Thus the conclusion is that equation \eqref{equation-condition.central} is equivalent to the condition that the determinants of the above systems vanish, i.e.
$$(\kappa-1)C_2C_4 - (\kappa+1)C_1C_3 = 0,$$
together with the statements (I) and (II) of the proposition, depending on the non-vanishing of the $C_i$'s.

Finally we show $(iii) \Rightarrow (i)$. We have already shown that $(iii)$ implies the vanishing of the second term in
$$T^2 = (a \omega + ct + du)^2 = (a^2 \omega^2 + cc^ty + dd^uz) + (cd^t \bar x \bar z +c^ud \bar y)s.$$
Thus it is enough to show that $(iii)$ also implies that the term $cc^ty + dd^uz$ is central. We have
\begin{align*}
& c^t = \bar x (c_1-c_2+c_3-c_4), \\
& d^u = d_1-d_2+d_3-d_4,
\end{align*}
therefore
\begin{align*}
cc^ty + dd^uz & = (c_1+c_2+c_3+c_4)(c_1-c_2+c_3-c_4) \bar x y + (d_1+d_2+d_3+d_4)(d_1-d_2+d_3-d_4) z \\
& = \bar x y (c_1^2 - c_2^2 + c_3^2 - c_4^2) + z(d_1^2 - d_2^2 + d_3^2 - d_4^2) + 2 \bar x y (c_1c_3 - c_2c_4) + 2z(d_1d_3 - d_2d_4).
\end{align*}
Let us now argue that $c_1c_3-c_2c_4 =d_1d_3-d_2d_4 = 0$. We have, on one hand,
\begin{align*}
c_1c_3 - c_2c_4 & = C_1C_3\gamma_1\gamma_3 - C_2C_4\gamma_2\gamma_4 \\
& = (1-\bar y^2)[(1+x^2+2x)C_1C_3 - (1+x^2-2x)C_2C_4] \\
& = 2x(1-\bar y^2)[(\kappa+1)C_1C_3 - (\kappa-1)C_2C_4] = 0.
\end{align*}
On the other hand,
\begin{align*}
d_1d_3 - d_2d_4 & = D_1D_3 \delta_1 \delta_3 - D_2D_4 \delta_2 \delta_4 \\
& = (1- \bar z^2)(D_1D_3 - 4 \omega^2 D_2D_4).
\end{align*}
In case (I), this equals
$$4(1-\bar z^2) \alpha \beta C_1C_4 [ (\kappa+1)(\kappa-1) - \omega^2] = 0,$$
and in case (II), we get
$$4(1-\bar z^2) \alpha \beta C_2C_3 [ (\kappa+1)(\kappa-1) - \omega^2] = 0.$$
We thus obtain the expression
$$T^2 = a^2 \omega^2 + (c_1^2 -c_2^2 +c_3^2 -c_4^2)\bar x y  + (d_1^2 -d_2^2 +d_3^2 - d_4^2)z.$$
The fact that all nine summands are in $\calZ$, i.e. that each $\gamma_i^2 \bar x y, \delta_i^2z \in \calZ$, follows from looking at the second last columns of Tables \eqref{table-first.basis} and \eqref{table-second.basis}. This proves the implication $(iii) \Rightarrow (i)$, and the proposition.
\end{proof}

We are now ready to present our main theorems.

\begin{theorem}\label{theorem-criterion.unit}
There exists a non-trivial unit $U \in k[G]$ of type $2$ if and only if the system of equations presented below have a solution
$$(A,C_1,C_2,C_3,C_4,D_1,D_2,D_3,D_4,\alpha,\beta) \in \calZ^{11}$$
different from $(\pm 1,0,0,0,0,0,0,0,0,0,0)$, such that $A \in k[\kappa_x,\kappa_y,\kappa_z,\xi]$ and $C_i\gamma_i, D_i\delta_i \in k[x^{\pm 1},y^{\pm 1},z^{\pm 1}]$.
\begin{align*}
\diamond \text{ } & (A^2-1)\omega_x^2 + C_1^2(\kappa_x+1)(\kappa_y-1)-C_2^2(\kappa_x-1)(\kappa_y+1) + C_3^2(\kappa_x+1)(\kappa_y+1)-C_4^2(\kappa_x-1)(\kappa_y-1) \\
& \phantom{(A^2-1)\omega_x^2} + 2D_1^2(\kappa_z-1)-2D_2^2\omega_x^2(\kappa_z+1) + 2D_3^2(\kappa_z+1)-2D_4^2\omega_x^2(\kappa_z-1) = 0, \\
\diamond \text{ } & (\kappa_x-1)C_2C_4 - (\kappa_x+1)C_1C_3 = 0, \\
\diamond \text{ } &  \text{ In case $C_1 \neq 0$ or $C_4 \neq 0$}, \begin{cases} D_1 = \beta (1 - \kappa_x) C_4, \\ D_2 = \alpha C_4, \\ D_3 = -\alpha (1 + \kappa_x) C_1, \\ D_4 = \beta C_1. \end{cases} \text{ In case $C_2 \neq 0$ or $C_3 \neq 0$}, \begin{cases} D_1 = -\beta (1 + \kappa_x) C_3, \\ D_2 = \alpha C_3, \\ D_3 = \alpha (1 - \kappa_x) C_2, \\ D_4 = \beta C_2. \end{cases} 
\end{align*}
\end{theorem}

\begin{remarks}\label{remarks-criterion.unit}
\text{ }
\begin{enumerate}[1),leftmargin=0.7cm]
\item We have preferred to write the equations in terms of the old notation $\omega_x,\kappa_x$ as it seems to be the most convenient form to reflect the explicit appearance of those terms in comparison with the other elements $\kappa_y,\kappa_z$. Nevertheless, we will continue to use the notation $\omega,\kappa$ for the proof of the theorem.
\item As stated, the criterion in Theorem \ref{theorem-criterion.unit} includes 6 equations and 11 variables. However, the substitutions given by the last $4$ equations reduce it to a system of 2 equations with $7$ variables $(A,C_1,C_2,C_3,C_4,\alpha,\beta)$. The second equation can be used to reduce further to a single equation with $6$ variables. We chose to state it in an ``extended'' form as it seems that it is the form most amenable to further analysis.
\item One can also use the conclusion of Corollary \ref{corollary-special.22.2} to prove Theorem \ref{theorem-criterion.unit}.
\end{enumerate}
\end{remarks}

\begin{proof}[Proof of Theorem \ref{theorem-criterion.unit}]
By Corollary \ref{corollary-special.22}, the existence of a non-trivial unit $U \in k[G]$ of type $2$ is equivalent to the existence of a non-trivial $\omega$-conjugate $T \in k[G]$ with $\omega$-degree $(2,2)$ of the form
\begin{equation}
T = A \omega + ct + du,
\end{equation}
where $A \in k[\kappa_x,\kappa_y,\kappa_z,\xi]$ and $c,d \in k[x^{\pm 1},y^{\pm 1}, z^{\pm 1}]$. We have already noted that non-triviality of $T$ implies that both $c,d \neq 0$.

So first suppose that such $T$ exists. In particular $T^2 = \omega^2 \in \calZ$, and so Proposition \ref{proposition-equivalence.central} tells us that all of the equations of the statement of the theorem, excluding the first one, follow. But the first one also follows since
\begin{align*}
\omega^2 = T^2 & = A^2 \omega^2 + (c_1^2-c_2^2+c_3^2-c_4^2)\bar x y + (d_1^2 - d_2^2 + d_3^2 - d_4^2)z \\
& = A^2 \omega^2 + C_1^2 \gamma_1^2 \bar x y - C_2^2 \gamma_2^2 \bar x y + C_3^2 \gamma_3^2 \bar x y - C_4^2 \gamma_4^2 \bar x y \\
& \qquad + D_1^2 \delta_1^2 z - D_2^2 \delta_2^2 z + D_3^2 \delta_3^2 z - D_4^2 \delta_4^2 z \\
& = A^2 \omega^2 + 4C_1^2(\kappa+1)(\kappa_y-1)-4C_2^2(\kappa-1)(\kappa_y+1)+4C_3^2(\kappa+1)(\kappa_y+1)-4C_4^2(\kappa-1)(\kappa_y-1) \\
& \qquad + 2D_1^2(\kappa_z-1)-8D_2^2\omega^2(\kappa_z+1)+2D_3^2(\kappa_z+1)-8D_4^2\omega^2(\kappa_z-1),
\end{align*}
which is the first equation after a change of variables $2C_i \ra C_i$ and $2D_2 \ra D_2, 2D_4 \ra D_4$. Note that the remaining equations of the theorem have been also changed accordingly under these changes of variables. Lastly, all of the variables $A,C_i,D_i,\alpha,\beta$ belong to $\calZ$, and moreover $A \in k[\kappa_x,\kappa_y,\kappa_z,\xi]$ and each $c_i = C_i \gamma_i, d_i = D_i \delta_i \in k[x^{\pm 1},y^{\pm 1},z^{\pm 1}]$ by Lemma \ref{lemma-equivalence.Laurent}. Finally, since both $c,d \neq 0$, at least some $c_i,d_j \neq 0$, and hence the solution obtained is different from the solution $(\pm 1,0,0,0,0,0,0,0,0,0,0)$.

Conversely, if such a solution of the stated system of equations exists, then the element
$$T = A \omega + \Big(\frac{C_1}{2}\gamma_1 + \frac{C_2}{2}\gamma_2 + \frac{C_3}{2}\gamma_3 + \frac{C_4}{2}\gamma_4 \Big) t + \Big(D_1\delta_1 + \frac{D_2}{2}\delta_2 + D_3\delta_3 + \frac{D_4}{2}\delta_4\Big) u$$
belongs to $k[G]$, and a straightforward computation using the equations of the statement of the theorem shows that $T^2 = \omega^2$, hence central. Clearly the $\omega$-degree of $T$ is $(2,2)$, and non-triviality of $T$ follows from the non-triviality of the solution obtained, i.e. different from $(\pm 1,0,0,0,0,0,0,0,0,0,0)$.
\end{proof}

We now present a variant of the previous theorem for anti-selfadjoint $\omega$-conjugates, in case the involution of $k$ restricts to the identity on its prime subfield.

\begin{theorem}\label{theorem-criterion.unitary}
Supose that the involution of $k$ restricts to the identity on its prime subfield. Then there exists a non-trivial unitary $U \in k[G]$ of type $2$ if and only if the system of equations from Theorem \ref{theorem-criterion.unit} has a solution
$$(A,C_1,C_2\xi,C_3\xi,C_4,D_1,D_2\xi,D_3\xi,D_4,\alpha,\beta) \in \calK^9 \times \calZ^2$$
different from $(\pm 1,0,0,0,0,0,0,0,0,0,0)$, such that $A \in k[\kappa_x,\kappa_y,\kappa_z]$ and $C_i\gamma_i, D_i\delta_i \in k[x^{\pm 1},y^{\pm 1},z^{\pm 1}]$.
\end{theorem}

For its proof, we first need some observations, which we will gather in the following lemma. Recall that $\calZ = \calK(\xi)$, which is an extension of degree $2$ over $\calK = k(\kappa_x,\kappa_y,\kappa_z)$.

\begin{lemma}\label{lemma-involution.center}
Suppose that the involution of $k$ is the identity. Then $\calZ^{\ast} = \calZ$, and moreover $z^{\ast} = z$ for any $z \in \calK$. Also, the following hold true.
\begin{enumerate}[i),leftmargin=0.7cm]
\item The element $z$ belongs to $\calK$ if and only if $z^{\ast} = z$.
\item The element $z \xi$ belongs to $\calK$ if and only if $z^{\ast} = - z$.
\end{enumerate}
\end{lemma}
\begin{proof}
The statement $\calZ^{\ast} = \calZ$ is clear. Note that $z^{\ast} = z$ for any $z \in \calK = k(\kappa_x,\kappa_y,\kappa_z)$ since $\kappa_x^{\ast} = \kappa_x, \kappa_y^{\ast}=\kappa_y$ and $\kappa_z^{\ast} = \kappa_z$.

Now given $z \in \calZ$, write it as $z = z_1 + \xi z_2$ for some $z_1,z_2 \in \calK$. Then $z^* = z_1 - \xi z_2$, and so the rest of the lemma follows straightforwardly.
\end{proof}

Finally we are in a position to prove Theorem \ref{theorem-criterion.unitary}.

\begin{proof}[Proof of Theorem \ref{theorem-criterion.unitary}]
This is a particular case of Theorem \ref{theorem-criterion.unit}, and the only thing which requires further study is the condition $T^{\ast} = -T$. We write
$$T = A \omega + C_1\gamma_1t + C_2\gamma_2t + C_3\gamma_3t + C_4\gamma_4t + D_1\delta_1u + D_2\delta_2u + D_3\delta_3u + D_4\delta_4u,$$
and we use the last columns of Tables \eqref{table-first.basis} and \eqref{table-second.basis}:
\begin{align*}
T^{\ast} & = A^{\ast} \omega^{\ast} + C_1^{\ast} (\gamma_1t)^{\ast} + C_2^{\ast} (\gamma_2t)^{\ast} + C_3^{\ast} (\gamma_3t)^{\ast} + C_4^{\ast} (\gamma_4t)^{\ast} + D_1^{\ast} (\delta_1u)^{\ast} + D_2^{\ast} (\delta_2u)^{\ast} + D_3^{\ast} (\delta_3u)^{\ast} + D_4^{\ast} (\delta_4u)^{\ast} \\
& = -A^{\ast} \omega - C_1^{\ast} \gamma_1t + C_2^{\ast} \gamma_2t + C_3^{\ast} \gamma_3t - C_4^{\ast} \gamma_4t - D_1^{\ast} \delta_1u + D_2^{\ast} \delta_2u + D_3^{\ast} \delta_3u - D_4^{\ast} \delta_4u
\end{align*}
Thus the condition $T^{\ast} = -T$ is equivalent, by Lemma \ref{lemma-involution.center}, to the conditions
$$A, C_1, C_2\xi, C_3\xi, C_4, D_1,D_2\xi,D_3\xi,D_4 \in \calK.$$
These are precisely the condition in the statement of Theorem \ref{theorem-criterion.unitary} which differs from the statement of Theorem \ref{theorem-criterion.unit}. The proof is complete.
\end{proof}

\section*{Acknowledgments}

The authors would like to thank Dawid Kielak for his helpful comments on the first part of the paper.

\bibliographystyle{plain}

\bibliography{bibliography}

\begin{thebibliography}{1}

\bibitem{AraGoodearl}
P.~Ara and K.~R. Goodearl.
\newblock The realization problem for some wild monoids and the {A}tiyah
  problem.
\newblock {\em Trans. Amer. Math. Soc.}, 369(8):5665--5710, 2017.

\bibitem{determinant}
D.~A. Craven and P.~Pappas.
\newblock On the unit conjecture for supersoluble group algebras.
\newblock {\em J. Algebra}, 394:310--356, 2013.

\bibitem{Tsen}
S.~Ding, M.-C. Kang, and E.-T. Tan.
\newblock Chiungtze {C}. {T}sen (1898--1940) and {T}sen's theorems.
\newblock {\em Rocky Mountain J. Math.}, 29(4):1237--1269, 1999.

\bibitem{CentralAlgebras}
B.~Farb and R.~K. Dennis.
\newblock {\em Noncommutative algebra}, volume 144 of {\em Graduate Texts in
  Mathematics}.
\newblock Springer-Verlag, New York, 1993.

\bibitem{Gardam}
G.~Gardam.
\newblock A counterexample to the unit conjecture for group rings.
\newblock {\em Ann. of Math. (2)}, 194(3):967--979, 2021.

\bibitem{Linnell88}
P.~H. Kropholler, P.~A. Linnell, and J.~A. Moody.
\newblock Applications of a new {$K$}-theoretic theorem to soluble group rings.
\newblock {\em Proc. Amer. Math. Soc.}, 104(3):675--684, 1988.

\bibitem{Luck}
W.~L\"{u}ck.
\newblock {\em {$L^2$}-invariants: theory and applications to geometry and
  {$K$}-theory}, volume~44 of {\em Ergebnisse der Mathematik und ihrer
  Grenzgebiete. 3. Folge. A Series of Modern Surveys in Mathematics [Results in
  Mathematics and Related Areas. 3rd Series. A Series of Modern Surveys in
  Mathematics]}.
\newblock Springer-Verlag, Berlin, 2002.

\bibitem{Murray}
A.~G. Murray.
\newblock More counterexamples to the unit conjecture for group rings.
\newblock {\em arXiv e-prints}, page arXiv:2106.02147, June 2021.

\bibitem{unboundedop}
Holger Reich.
\newblock {\em Group von Neumann Algebras and Related Algebras}.
\newblock PhD thesis, Georg-August-Universität zu Göttingen, 1998.

\end{thebibliography}

\end{document}